\documentclass[english,a4paper]{amsart}

\usepackage{latexsym,amsthm,amscd}

\usepackage[all]{xy}
\usepackage{rotating}
\usepackage[utf8]{inputenc}
\usepackage[T1]{fontenc}
\usepackage{babel}
\usepackage{amsmath,leftidx}
\usepackage{accents}
\usepackage{amssymb}
\usepackage{amsfonts}
\usepackage{enumerate} 
\usepackage{caption}
\usepackage[svgnames]{xcolor}
\usepackage{array}
\usepackage{stmaryrd}
\usepackage{tikz}
\usetikzlibrary{patterns}
\usepgflibrary{decorations.pathmorphing}
\usetikzlibrary{decorations.pathmorphing}
\usetikzlibrary{calc,shapes.arrows,decorations.pathreplacing}
\usetikzlibrary{shadows}
\usetikzlibrary{positioning}

\newtheorem{theorem}{Theorem}[section]
\newtheorem{lemma}[theorem]{Lemma}
\newtheorem{cor}[theorem]{Corollary}
\newtheorem{question}[theorem]{Question}
\newtheorem{definition}[theorem]{Definition}

\newtheorem{notation}[theorem]{Notation}
\newtheorem{proposition}[theorem]{Proposition}
\newtheorem{remark}[theorem]{Remark}

\newtheorem*{T1}{Theorem~\ref{theorem.main}}

\newcommand{\R}{\mathbb{R}}

\newcommand{\E}{\mathcal{E}}

\DeclareMathOperator{\diam}{diam}

\DeclareMathOperator{\mesh}{mesh}

\author{Silvère Gangloff}
\address[S. Gangloff]{AGH University of Science and Technology, Faculty of Applied
	Mathematics, al.
	Mickiewicza 30, 30-059 Krak\'ow, Poland
}
\email[S.~Gangloff]{sgangloff@agh.edu.pl}

\author{Piotr Oprocha}
\address[P.~Oprocha]{AGH University of Science and Technology, Faculty of Applied
	Mathematics, al.
	Mickiewicza 30, 30-059 Krak\'ow, Poland
	-- and --
	Centre of Excellence IT4Innovations - Institute for Research and Applications of Fuzzy Modeling, University of Ostrava, 30. dubna 22, 701 03 Ostrava 1, Czech Republic.}
\email[P.~Oprocha]{oprocha@agh.edu.pl}

\title[Cantor dynamical system is slow iff all finite orbits are attracting]{A Cantor dynamical system is slow if and only if all its finite orbits are attracting}

\numberwithin{equation}{section}
\begin{document}
	
	\begin{abstract}
		In this paper we completely solve the problem of when a Cantor dynamical system $(X,f)$ 
		can be embedded in $\mathbb{R}$ with vanishing derivative. 
		For this purpose we construct a refining sequence of
		marked clopen partitions of $X$ 
		which is adapted to a dynamical system of this kind. 
		It turns out that there is a huge class of such systems.
	\end{abstract}

		\maketitle
		
\section{Introduction}

We say that a dynamical system $(X,f)$ can be \textit{embedded in the real line with vanishing derivative} when there exists a differentiable function $g\colon \R \to\R $, a closed set $Z \subset \mathbb{R}$ invariant for $g$ and a homeomorphism $\psi\colon X \rightarrow Z$ such that $g|_Z =\psi \circ f \circ \psi^{-1}$ and $g'|_Z \equiv 0$.
In this paper we completely solve the following problem:

\begin{question}\label{q:cantor}
What maps acting on a Cantor set can be embedded in the real line with vanishing derivative?
\end{question}
	
The well known Banach fixed point theorem implies that a contraction on compact metric space must have a fixed point.
Therefore searching for an example of map $f$ on the real line with a closed invariant set without fixed points on which $f'\equiv 0$
seems a task doomed to failure. In fact, it was suspected by Edrei in his paper from 1952 \cite{Edrei}
that even a weaker condition cannot be satisfied. Generally speaking, Edrei conjectured that any map on a compact set which locally does not increase distances must be an isometry.
Soon after \cite{Williams} Williams provided examples of maps which are not local isometries at some points,
answering the original question of Edrei. However all these examples have isolated points and some of them also have fixed points.
Since then relations between local shrinking and periodic points remained unclear, until very recently.
In 2016 Jasi\'nski and Ciesielski constructed in \cite{ciesielski} an embedding of 2-adic odometer in the real line.
The embedding was obtained by direct application of Jarn\'\i{}k theorem together with very delicate construction of
a metric on Cantor set leading to derivative zero. This technique was in huge part relying on a clever representation of
the 2-adic group defining odometer. This technique was then extended in \cite{BKOAHP} to all odometers, together with some other 
carefully constructed examples such as an attractor-repellor pair or transitive non-minimal Cantor system. 
Still, it was not clear how much vanishing derivative is correlated with the existence of a metric making the map
an isometry or at least having entropy zero. At this point it was expected that such an embedding may not exist for expansive maps (in particular subshifts)
since these systems have divergence of orbits hidden in the dynamics (see discussion in \cite{BKO}). Then in~\cite{BKO}, J.P.Boro\'nski, J.Kupka and P.Oprocha brought a surprising answer to this, showing 
that every	minimal dynamical system on a Cantor set can be embedded in the real line with vanishing derivative everywhere. The proof of this result makes use of J.-M.Gambaudo 
and M.Martens~\cite{Gambaudo-Martens}
representation of minimal dynamical systems on the Cantor set with graph 
coverings which satisfy certain properties. These representations are to some extent similar to representations of odometers,
sharing the property that there is a unique vertex at which all the cycles intersect. This property was crucial in the main proof.
	
It was clear from examples in \cite{BKOAHP} that minimality does not characterize the class of Cantor systems which can be embedded in the real line with vanishing 
derivative. On the other hand, it was known that periodic points are in many cases
an obstacle for the required embedding. As a consequence the natural class to consider was the one of aperiodic systems.
It is known that aperiodic 
systems on a Cantor set can be represented 
with Bratelli-Vershik diagrams~\cite{Meydinets}. There are also techniques (e.g. see \cite{Shinomura.odometers.extensions})
which allow to obtain graph coverings 
representation out of Kakutati-Rohlin towers, rooting the Bratelli-Vershik representation for minimal systems on the Cantor set.
Despite the fact that representations of aperiodic systems do not have the convenient structure of Gambaudo-Martens representations and can consist of cycles with numerous points of intersections, we managed to
describe them in a way suitable for our needs. One valuable tool in our research was a deep understanding of aperiodic systems reflected in recent results (e.g. see \cite{Downarowicz}). The paper is almost completely devoted to proving the following theorem, providing a complete answer to Question~\ref{q:cantor}.
	
\begin{theorem}\label{theorem.main}
A dynamical system $(X,f)$ on a Cantor set $X$ can be embedded in the real line with vanishing derivative if and only if all finite orbits 
of $(X,f)$ are attractors. 
\end{theorem}

In fact it is possible to state this result in a slightly more general way.	
If $X$ is a finite set then it is obvious that it can be embedded in the real line with vanishing derivative.
If $X$ is infinite but zero dimensional then we can present $X$ as $X=C\cup R$ where $C$ is a Cantor set and $R$ is at most countable and consists in isolated points.
We can then replace each isolated point $x$ by a Cantor set $C_x$ and define $g(y)=f(x)$ for every $y\in C_x$. Additionally, the limit set of any $y\in C_x$
is the same as the one of $x$. This way we obtain
a map $g$ acting on Cantor set such that $g|_X=f$. We can then apply Theorem~\ref{theorem.main} to $g$, obtaining the following.

\begin{theorem}\label{cor.main}
Any dynamical system $(X,f)$ on a zero-dimensional compact set $X$ can be embedded in the real line with vanishing derivative if and only if all finite orbits of $(X,f)$ are attractors. 
\end{theorem}

In \cite{ciesielski} the authors introduced the property of
\textit{locally radially shrinking} defined as follows:
\begin{itemize}
	\item[(LRS)] for every $x\in X$ there exists an $\epsilon_x>0$ such that $d(x,y)<\epsilon_x$ implies $d(f(x),f(y))<d(x,y)$ for all $y\neq x$.
\end{itemize}
It is clear from this definition that if a map $f$ with (LRS) has a periodic point, then its orbit is attracting. This combined with Theorem~\ref{theorem.main} shows that on zero-dimensional compact sets, the class of maps with (LRS) and the class of maps which can be embedded in real line with vanishing derivative are exactly the same. This reveals another unexpected connection between shrinking and dynamics. It is also worth emphasizing that being attractor is a topological property, while (LRS) depends on metrics.

This article is organized as follows: Section~\ref{section.graph.coverings} is an 
	exposition of the graph coverings representation of Cantor dynamical systems.
	For the reader's convenience we also provide proofs of some elementary facts, in 
	order for the article to be self-contained.
	Section~\ref{section.finite.orbits} contains a construction of graph covering
	representations adapted to Cantor 
	dynamical systems which have only attracting finite orbits, 
	and Section~\ref{section.proof} contains a proof of Theorem~\ref{theorem.main}.
	
\section{Zero-dimensional dynamical systems and graph coverings\label{section.graph.coverings}}
	
	Let us start with basic definitions.
		A topological space is called \textbf{zero-dimensional} when it has a base which consists 
		of clopen (open and closed) sets. A \textit{Cantor set} is any zero-dimensional compact metric space without isolated points. 
		It is well known that all Cantor sets are homeomorphic.
	
		In this paper $(X,d)$ is a fixed Cantor set and $f\colon X \rightarrow X$ is a continuous 
		function. We will later impose additional conditions on $f$.
	For short, the pair $(X,f)$ will be called a \textit{Cantor system} or a \textit{dynamical system} (on the Cantor set $X$).

	\begin{notation}
	For every set $\mathcal{U}$ of open subsets of $X$ (in particular a 
	partition or a subset of a partition), we denote $\mathcal{E}(\mathcal{U})$ the union of 
	its elements.
	\end{notation}
	
	\begin{notation}
	For the remainder of this paper, we fix a sequence of partitions $(\mathcal{U}^0_n)_{n \ge 1}$ of $X$ into clopen sets satisfying the
	condition:
	\begin{eqnarray}
	&\forall\; n\ge 1,\; \forall\; u\in \mathcal{U}^0_{n},\; u=\mathcal{E}(\{v\in \mathcal{U}^0_{n+1} : v \subset u\}) \qquad \text{ and }\label{eq:star}\\
	&\displaystyle{\lim_{n\to\infty}\mesh(\mathcal{U}^0_n)} =0.\nonumber
	\end{eqnarray}	
	\end{notation}
\begin{remark}
	 For every finite clopen partition $\mathcal{V}$ of $X$, there exists some integer $n \ge 1$ such that every element of $\mathcal{V}$ is the union of some elements in $\mathcal{U}^0_n$.

\end{remark}
	
	We reproduce in Section~\ref{section.general}, for completeness, the characterization of Cantor systems in terms
	of graph coverings that one can find in~\cite[Theorem 3.9]{Shinomura}. We prove a characterization of these representations for aperiodic Cantor systems in Section~\ref{section.aperiodic}. 
	
	\subsection{General formulation\label{section.general}}
		
	\subsubsection{Clopen partitions and graphs}
	
	In the following we will use finite clopen partitions and finite graphs in order to represent the 
	behavior of the dynamical system $(X,f)$. Preliminary notations 
	are introduced in this section.
	
	\begin{notation}
	For every finite clopen partition $\ \mathcal{U}$ of $X$, we will denote by $\ G(\mathcal{U})$ 
	the finite directed graph whose vertex set is $\ \mathcal{U}$ and whose edges are the 
	pairs $(u,v) \in \mathcal{U}^2$ such that there exists
	$x \in u$ with $f(x) \in v$.
	\end{notation}
	
	\begin{definition}
	Let us consider two directed 
	graphs $G$ and $G'$ whose vertex sets are respectively $V$ and $V'$ 
	and respective sets of edges are $E$ and $E'$. A \textbf{graph morphism} from $G$ to $G'$ is 
	a function $\pi : V \rightarrow V'$ such that whenever $(u,v) \in E$, 
	$(\pi(u),\pi(v)) \in E'$.
	\end{definition}
	
	\begin{definition}
	For two finite clopen partitions $\mathcal{U},\mathcal{V}$, we say that $\mathcal{V}$
	\textbf{refines} $\mathcal{U}$ when for all $v \in \mathcal{V}$ there exists 
	$u \in \mathcal{U}$ such that $v \subset u$. This relation will be denoted 
	$\mathcal{V} \prec \mathcal{U}$. 
	\end{definition}
	
	For two finite clopen partitions $\mathcal{U},\mathcal{V}$, we will also denote 
	\[\mathcal{U} \vee \mathcal{V} = \left\{ u \cap v \ : \ u \in \mathcal{U}, v \in \mathcal{V}\right\},\]
	which is another finite clopen partition refining both $\mathcal{U}$ and $\mathcal{V}$.
	
	\begin{notation}
	Let us consider two finite clopen partitions $\mathcal{U},\mathcal{V}$ of $X$ such that 
	$\mathcal{V} \prec \mathcal{U}$. We will denote $\pi_{\mathcal{U}}^{\mathcal{V}}$ the 
	graph morphism from $G(\mathcal{V})$ to $G(\mathcal{U})$ such that for all 
	$u \in \mathcal{V}$, $u \subset \pi_{\mathcal{U}}^{\mathcal{V}}(u)$.
	\end{notation}
	
	\subsubsection{From zero-dimensional systems to graph coverings\label{section.forward}}

	\begin{definition}\label{definition.graph.covering}
		The \textbf{graph covering representation} of the dynamical system $(X,f)$ relative 
		to a sequence of finite clopen partitions $\mathbb{U}= (\mathcal{U}_n)_{n \ge 1}$ which satisfy the condition \eqref{eq:star}, denoted by $\mathcal{G}(\mathbb{U},f)$, is the data of two sequences $(G_n)_{n \ge 1}$ and 
		$(\pi_n)_{n \ge 1}$ such that:
		\begin{enumerate}[(i)]
			\item  for all $n \ge 1$, $G_n = G(\mathcal{U}_n)$; 
			\item for all $n \ge 1$, $\pi_n = \pi_{\mathcal{U}_{n}}^{\mathcal{U}_{n+1}}$. 
		\end{enumerate}
	For all $m \ge n$ we will also set $\pi_{n,m} = \pi_n \circ ... \circ \pi_m$.
We denote by $V_{\mathbb{U},f}$ the following set: 
		\[V_{\mathbb{U},f} = \displaystyle{\left\{(u_n)_{n \ge 1} \in \prod_{n \ge 1} V_n : \forall n \ge 1, \pi_n (u_{n+1}) = u_n\right\}}.\]
	\end{definition}

We may view $V_{\mathbb{U},f}$ as an inverse limit defined by the spaces $V_n$ together with bonding maps $\pi_n$.
	
	\begin{lemma}\label{lemma.no.double.image}
		For each $\textbf{u} \in V_{\mathbb{U},f}$, there exists a unique $\textbf{v} \in V_{\mathbb{U},f}$ such that for all $n \ge 1$, $(\textbf{u}_n,\textbf{v}_n) \in E_n$. (we will set $\underline{f}_{\mathbb{U}}(\textbf{u}):=\textbf{v}$).
	\end{lemma}
	
	\begin{proof}
		The proof is standard. We give it for completeness.
		\begin{itemize}
			\item \textbf{Existence:} Let us consider $\textbf{u} \in V_{\mathbb{U},f}$, and denote $x$ the 
			unique element of $X$ such that for all $n \ge 1$, $x \in \textbf{u}_n$. 
			Let us fix $\textbf{v}$ such that for all $n \ge 1$, $f(x) \in \textbf{v}_n$. By definition for all $n \ge 1$, $(\textbf{u}_n,\textbf{v}_n) \in E_n$ and since $f(x)\in \textbf{v}_{n+1}, \textbf{v}_{n}$
			we have $\textbf{v}_{n+1}\subset \textbf{v}_{n}$ and so $\pi_n (\textbf{v}_{n+1}) = \textbf{v}_n$. In particular $\textbf{v}\in V_{\mathbb{U},f}$.
			\item \textbf{Uniqueness:} Moreover let 
			us consider $\textbf{v}'\in V_{\mathbb{U},f}$ such that 
			for all $n \ge 1$, $(\textbf{u}_n,\textbf{v}'_n) \in E_n$. For all $n \ge 1$, by definition 
			of the graph $G_n$, there exists $x_n \in \textbf{u}_n$ such that $f(x_n) \in \textbf{v}'_n$. The 
			sequence $(x_n)_n$ converges towards $x$ and since $f$ is continuous $f(x_n)$ converges towards $f(x)$. As a consequence for all $k \ge 1$ there exists some $l$ such that for all $m \ge l$, 
			$\textbf{v}'_m\subset \textbf{v}_k$. In particular $\textbf{v}'_k\cap \textbf{v}_k=\pi_{k,m}(\textbf{v}'_{m+1})\cap \textbf{v}_k\neq \emptyset$, which implies
			$\textbf{v}'_k=\textbf{v}_k$. This proves $\textbf{v}'=\textbf{v}$, completing the proof.
		\end{itemize}
	\end{proof} 
	
	\begin{definition}
		Let us denote by $\varphi_{\mathbb{U},f}\colon X \rightarrow V_{\mathbb{U},f}$ (or 
		simply $\varphi$ when there is no ambiguity) the function such that for all $x \in X$, 
		$\varphi_{\mathbb{U},f}(x)$ is the unique (since $\text{mesh}(\mathcal{U}_n) \rightarrow 0$) sequence $\textbf{u} \in V_{\mathbb{U},f}$ such that for all $n \ge 1$, 
		$x \in \textbf{u}_n$.
	\end{definition}
	
	\noindent The following is straightforward:
	
	\begin{proposition}\label{prop:conj}
		The map $\varphi_{\mathbb{U},f}$ is a homeomorphism and $\underline{f}_{\mathbb{U}} = \varphi_{\mathbb{U},f} \circ f \circ \varphi_{\mathbb{U},f}^{-1}$.
	\end{proposition}
	
	\subsubsection{From graph coverings to zero-dimensional systems}
	
	Reciprocally, let us consider a sequence of finite graphs $\textbf{G}=(G_n)_{n \ge 1}$ and 
	a sequence $\mathbf{\pi}=(\pi_n)_{n \ge 1}$ of surjective graph morphisms $\pi_n : V_{n+1} \rightarrow V_n$, and assume that for all $u \in V_n$ there exists $v \in V_n$ such that $(u,v) \in E_n$,
	where $V_n$ is the vertex set of $G_n$ and $E_n$ its edge set. Let us denote $V_{\textbf{G},\pi}$ the set 
	\[V_{\textbf{G},\pi} = \displaystyle{\left\{(u_n) \in \prod_{n \ge 1} V_n : \forall n \ge 1, \pi_n (u_{n+1}) = u_n\right\}}.\]
	This set is a metric space with the metrization of Tychonoff product of 
	discrete topologies on the finite sets $V_n$, $n \ge 1$. Let us also assume that for all $\textbf{u} \in V_{\textbf{G},\pi}$, there exists a unique $\textbf{v} \in V_{\textbf{G},\pi}$ 
	such that $(\textbf{u}_n,\textbf{v}_n) \in E_n$  for all $n \ge 1$ (it is in particular 
	the case when for every $n$ that if $(u,v), (u,w)\in E_{n+1}$ then $\pi_n(v)=\pi_n(w)$).
	Denoting $f_{\textbf{G},\pi}(\textbf{u})$ 
	this sequence $\textbf{v}$, we have the following:
	
	\begin{lemma}
		The map $f_{\textbf{G},\pi} : V_{\textbf{G},\pi} \rightarrow V_{\textbf{G},\pi}$ is continuous (as a consequence $(V_{\textbf{G},\pi},f_{\textbf{G},\pi})$ is a dynamical system on a zero-dimensional compact metric space).
	\end{lemma}
	
	\begin{proof}
		Let us consider some $\textbf{u} \in V_{\textbf{G},\pi}$, a sequence $(\textbf{u}^k)_{k \ge 0}$ such that 
		$\textbf{u}^k \rightarrow \textbf{u}$, and a subsequence $(\textbf{v}^k)_{k \ge 0}$ of
		this one such that
		 $(f_{\textbf{G},\pi}(\textbf{v}^k))_k$ converges towards some $\textbf{v}$.
		For all $n \ge 1$ there exists some $k_n$ such that for all $k \ge k_n$ and all $l \le n$, $\textbf{u}^k_l = \textbf{u}_l$ and $f_{\textbf{G},\pi}(\textbf{u}^k_l) = \textbf{v}_l$. As a consequence for all $l \le n$, 
		$(\textbf{u}_l,\textbf{v}_l) \in E_l$. Since this is satisfied for all $n \ge 1$, we have $\textbf{v} = f_{\textbf{G},\pi}(\textbf{u})$. As a consequence every subsequence of $(f_{\textbf{G},\pi}(\textbf{u}^k))_k$ converges towards $f_{\textbf{G},\pi}(\textbf{u})$, so $(f_{\textbf{G},\pi}(\textbf{u}^k))_k$ converges towards $f_{\textbf{G},\pi}(\textbf{u})$. We proved that $f_{\textbf{G},\pi}$ is continuous. 
	\end{proof}
	
	\begin{definition}
		A \textbf{telescoping} of some $(\textbf{G},\pi)$ is a tuple $(\textbf{G}',\pi')$ such that 
		there exists a sequence of integers $\textbf{n}=(n_k)_{k \ge 0}$ with $G'_k = G_{n_k}$ and $\pi'_k = \pi_{n_k,n_{k+1}}$ for all $k \ge 0$.
	\end{definition}
	
	From this definition it is straightforward that for all $\textbf{u} \in V_{\textbf{G}',\pi'}$ there is a unique $\textbf{v}$ such that $(\textbf{u}_k,\textbf{v}_k)$ is an edge of $G'_k$ for all $k \ge 0$, provided that  uniqueness of edges holds for $V_{\textbf{G},\pi}$. The following is also obvious. 
	
	\begin{lemma}\label{lemma.telescoping}
		For each telescoping $(\textbf{G}',\pi')$ of $(\textbf{G},\pi)$, 
		$(V_{\textbf{G}',\pi'},f_{\textbf{G}',\pi'})$ is conjugated to $(V_{\textbf{G},\pi},f_{\textbf{G},\pi})$.
	\end{lemma}
	
	\subsection{Aperiodic Cantor systems\label{section.aperiodic}}

Let $\mathcal{G}(\mathbb{U},f)= ((G_n)_{n \ge 1},(\pi_n)_{n \ge 1})$ 
	be a graph coverings representation of $f$ associated with a sequence of partitions $\mathbb{U}$.
	We assume that for all $n \ge 1$, if $(u,v) \in E_{n+1}$ and $(u,v') \in E_{n+1}$, 
	then $\pi_n (v) = \pi_n (v')$. Such a sequence of partitions always exists, since the space is zero-dimensional. 
	
	In the following, for every graph considered we will designate by \textbf{circuit} 
	any cycle in this graph which is minimal for the inclusion (considering a \textit{cycle} as a sequence of edges). The length of a circuit is the number 
	of edges it contains. 
	
	\begin{lemma}\label{lemma.periodic.points}
		A point $x \in X$ is periodic for $f$ with period $k \ge 1$ ($f^k(x)=x$)
		if and only if there exists some $m \ge 1$ and a sequence $(c_n)_{n \ge m}$ such that for all $n \ge m$, $c_n$ is a circuit of length $k$ in $G_n$, $\pi_n(c_{n+1}) = c_n$ and $\varphi(x)_n$ is a vertex in  $c_n$.
	\end{lemma}
	
	\begin{proof}
		\begin{itemize} 
			\item[$(\Rightarrow)$:] Let us consider some $x \in X$ such that
			$f^k (x) = x$.  Then Proposition~\ref{prop:conj} implies that $\underline{f}^k_{\mathbb{U}}(\varphi(x)) = \varphi(x)$. 
			Moreover there exists some $m \ge 1$ such that for all $n \ge m$, the points $f^l(x), l < k$, 
			belong to distinct elements of the partition $\mathcal{U}_{n}$ which form a cycle $c_{n}$ of length $k$ in $G_{n}$. We have directly that $c_{n+1}$ is mapped to $c_n$ by $\pi_n$. Moreover by definition $\varphi(x)_n$ is an element of $c_n$ for all $n \ge m$. 
			 If cycles $c_n$ are not minimal cycles, that is, they contain a cycle of smaller period, then it is not hard to see that the period of
			 $x$ is smaller than $k$, which is a contradiction.
			\item[$(\Leftarrow)$:] Reciprocally considering some sequence of circuits such as in the statement of the lemma, let us consider $\textbf{u}^0_n, ... \textbf{u}^{k}_n$ elements of the circuit $c_n$ 
			such that $\textbf{u}^0_n = \varphi(x)_n$, for all $l \le k-1$, $(\textbf{u}^l_n,\textbf{u}^{l+1}_n) \in E_n$ and $(\textbf{u}^k_n,\textbf{u}^{0}_n) \in E_n$. 
			Let $\textbf{w}=\varphi(f(x))$ and observe that $(\textbf{u}^0_n,\textbf{w}_n) \in E_n$ which implies (by hypothesis) that
			$$
			\textbf{u}^1_{n-1}=\pi_{n-1}(\textbf{u}^1_n)=\pi_{n-1}(\textbf{w}_n)=\textbf{w}_{n-1}.
			$$
			Therefore $\textbf{w}=\textbf{u}^1$ and by the same reasoning we have 
			that
			$\textbf{u}^{l+1} = \varphi(f^l(x))$ for all $l \le k$. 
			Since $\varphi(x)_n = \textbf{u}^k_n$ for all $n$, $\underline{f}^k_{\mathbb{U}} (\varphi(x)) = \varphi(x)$, which implies that $f^k(x)=x$. Since each $c_n$ is a circuit, the period of $x$ cannot be smaller than $k$.
		\end{itemize}
	\end{proof}
		
		\begin{notation}
			For all $n \ge 1$, let us denote by $\nu_n(\mathbb{U},f)$ the minimal length of a circuit
			in $G_n$. When the partition is clear from the context, we simply write $\nu_n(f)$.
		\end{notation}
		
	\begin{theorem}
	Assume that a sequence of partitions $\mathbb{U}$ satisfies the condition \eqref{eq:star}. Then the dynamical system $(X,f)$ is aperiodic if and only if $\nu_n(\mathbb{U},f) \rightarrow +\infty$.
	\end{theorem}
	
	\begin{proof}
		\begin{itemize}
			\item[$(\Rightarrow)$:] Using Lemma~\ref{lemma.telescoping}, it is sufficient to prove that for all $n$ there exists $m \ge n$ such that all the circuits
			in $G_{m+1}$ are mapped through $\pi_{n,m}$ to a concatenation of at least two circuits in $G_n$. Let us assume ad absurdum that there exists $n$ such that for all $m \ge n$, there exists a circuit in $G_m$ which is mapped injectively onto a circuit in $G_n$, i.e. its image is not a concatenation of two or more circuits in $G_n$.
			For this $n$ enumerate circuits in $G_n$, say $c_1^n,\ldots, c_{k_n}^n$ and let $r$ be the length of the longest of them.
			Then for all $m > n$ enumerate circuits in $G_m$ of length at most $r$, say $c_1^m,\ldots, c_{k_m}^m$.
			For each $m > n$ there exists a sequence $i_m^{m-1},\ldots,i_m^n$ such that for all $s \in \{n, \ldots, m-1\}$,  $\pi_{s,m-1}(c_{i_m^{m-1}})=c_{i_m^s}$.
			Denote by $x^{(m)} \in \Pi_{i=n}^\infty \{1,\ldots, k_i\}$ the sequence such that $(x^{(m)})_s=i_m^s$ for $s<m$ and $(x^{(m)})_s=1$ for $s \geq m$.
			By extracting a subsequence if necessary, we may assume that $(x^{(m)})$ converges to some sequence $x$.
			Then for every $s > n$ and $s' > s$ sufficiently large we have 
			$$
			\pi_{s-1}(c_{x_{s}})=\pi_{s-1}(\pi_{s,s'-1}(c_{x_{s'}}))=\pi_{s-1,s'-1}(c_{x_{s'}})=c_{x_{s-1}}.
			$$
			We constructed an inverse sequence of circuits, hence by Lemma~\ref{lemma.periodic.points}
			the system $(X,f)$ has a periodic point, which was assumed to be false.
			\item[$(\Leftarrow)$:] Reciprocally, if $\nu_n(f) \rightarrow +\infty$, it is clear that the conditions of Lemma~\ref{lemma.periodic.points} can not be satisfied for any sequence of 
			circuits $(c_n)$, thus $(X,f)$ has no periodic point.
		\end{itemize}
	\end{proof}
	
	\begin{remark}
		If $(X,f)$ is aperiodic then there is a subsequence of the sequence of partitions $\mathbb{U}$ such that in the associated telescoping $(\textbf{G}',\pi')$ of $(\textbf{G},\pi)$, the image $\pi_n'(c)$ of any circuit $c$ in $G'_{n+1}$ is a concatenation of at least two circuits in $G'_{n}$.
	\end{remark}
	
	\begin{remark}
		Similar to the graph coverings representation of Gambaudo and Martens \cite{Gambaudo-Martens}, we just proved a  
		characterization for aperiodicity. 
While in general we may not hope for a representation with all cycles intersecting at a unique vertex as in \cite{Gambaudo-Martens},
	we will show that some ``special'' vertices still exists, and some other useful properties can be required.
	\end{remark}
	
	\section{Purely attracting zero-dimensional systems\label{section.finite.orbits}}
	This section contains the main changes compared to the initial construction 
	of~\cite{BKO}.
	In order to prove the Theorem~\ref{theorem.main}, we will 
	consider graph coverings representations in the case of zero-dimensional systems 
	with only attracting finite orbits, which is 
	more complex than the minimal ones. In particular we will need to 
	define particular finite clopen partitions - that we call supercyclical - 
	which are adapted to the case considered, by discriminating two disjoint parts,
	one of which consists in a neighborhood of all the finite orbits whose length is smaller 
	than a certain integer - called attracted part. Moreover we will need to enrich the 
	graph coverings representation - exposed in Section~\ref{section.general} - 
	on sequences of refining supercyclical partitions with some markers which satisfy some 
	structural conditions. 
	
	We define attracting orbits in Section~\ref{section.attracting.orbits} and supercyclical 
	partitions in Section~\ref{section.supercyclical.partitions}. 
	We define then some operations on supercyclical partitions which act separately on the 
	attracted part and the other one - called supercyclical.
	In Section~\ref{section.attracted.operations} we deal with the attracted part. 
	In the following sections we deal with the supercyclical part: 
	in Section~\ref{section.krieger} we expose 
	the construction of Krieger markers in the supercyclical part and in 
	Section~\ref{section.marking} we 
	expose how to mark supercyclical partitions such that the set of markers forms an 
	acyclical cut of the corresponding graph using Krieger markers, in a way that 
	markers can not be mapped to 
	other markers by iteration of the graph morphisms. 
	Section~\ref{section.rectification} contains 
	a description of a procedure in order to ensure that in every partition of the sequence, 
	the divergent points - which have at least two outgoing edges in the graph - 
	coincide with markers. \bigskip 
	
	Before entering into the exposition of this part, let us make some remark on 
	the graph coverings representations of Cantor systems. It can happen that some 
	of the edges between vertices of a certain level $G_n$ are not 
	representing accurately the dynamical behavior of the system. Let us illustrate 
	this on an example: consider for instance the system which has a graph 
	coverings representation as on Figure~\ref{figure.example}. Although in the first level 
	$G_1$ the two circuits have nonempty intersection, one can see that the system still 
	consists in two disjoint periodic orbits. In fact one could separate the two circuits
	in the first level while the graph coverings would still represent the 
	same system. It is possible to construct a procedure that transforms a graph 
	coverings representation of a system into a similar representation of the same system 
	but without this phenomenon, which can be inconvenient in the proofs using 
	graph coverings (in particular for Theorem~\ref{theorem.main} in the present text). 
	We had to deal with a couple other problems of this type. Although it is in principle 
	possible to deal with them by working directly on a graph coverings representation obtained 
	with a simple sequence of refining clopen partitions, it appeared a lot more 
	efficient and simple to work on constructing a suitable sequence of partitions before 
	considering the graph coverings representation associated with it.

\begin{figure}[h!]
\begin{center}
\begin{tikzpicture}[scale=0.1]
\draw[fill=gray!98] (0,-4) circle (7pt);
\draw[fill=gray!98] (0,4) circle (7pt);
\draw[fill=gray!98] (-4,0) circle (7pt);
\draw[fill=gray!98] (4,0) circle (7pt);
\draw[fill=gray!98] (2.8,2.8) circle (7pt);
\draw[fill=gray!98] (2.8,-2.8) circle (7pt);
\draw[fill=gray!98] (-2.8,2.8) circle (7pt);
\draw[fill=gray!98] (-2.8,-2.8) circle (7pt);
\draw[-latex] (0,4) -- (-2.8,2.8);
\draw[-latex] (-2.8,2.8) -- (-4,0);
\draw[-latex] (2.8,2.8) -- (0,4);
\draw[-latex] (4,0) -- (2.8,2.8);
\draw[latex-] (4,0) -- (2.8,-2.8);
\draw[-latex] (0,-4) -- (2.8,-2.8);
\draw[latex-] (0,-4) -- (-2.8,-2.8);
\draw[-latex] (-4,0) -- (-2.8,-2.8);

\draw[-latex]  (15,-8) to[out=45,in=-45] (14,0);

\node[scale=0.9] at (21,-5) {$\pi_1$};
\node[scale=0.9] at (-17,0) {$G_1$};
\node[scale=0.9] at (-17,-12) {$G_2$};
\node[scale=0.9] at (-17,-24) {$G_3$};

\draw[-latex]  (-7,-8) to[out=135,in=-135] (-6,0);

\begin{scope}[yshift=-32cm]
\node at (4,0) {$\vdots$};
\end{scope}

\begin{scope}[xshift=8cm]
\draw[fill=gray!98] (0,-4) circle (7pt);
\draw[fill=gray!98] (0,4) circle (7pt);
\draw[fill=gray!98] (-4,0) circle (7pt);
\draw[fill=gray!98] (4,0) circle (7pt);
\draw[fill=gray!98] (2.8,2.8) circle (7pt);
\draw[fill=gray!98] (2.8,-2.8) circle (7pt);
\draw[fill=gray!98] (-2.8,2.8) circle (7pt);
\draw[fill=gray!98] (-2.8,-2.8) circle (7pt);
\draw[-latex] (0,4) -- (-2.8,2.8);
\draw[-latex] (-2.8,2.8) -- (-4,0);
\draw[-latex] (2.8,2.8) -- (0,4);
\draw[-latex] (4,0) -- (2.8,2.8);
\draw[latex-] (4,0) -- (2.8,-2.8);
\draw[-latex] (0,-4) -- (2.8,-2.8);
\draw[latex-] (0,-4) -- (-2.8,-2.8);
\draw[-latex] (-4,0) -- (-2.8,-2.8);
\end{scope}

\begin{scope}[yshift=-12cm]

\draw[-latex]  (15,-10) to[out=45,in=-45] (15,-2);
\node[scale=0.9] at (21,-7) {$\pi_2$};
\draw[-latex]  (-7,-10) to[out=135,in=-135] (-7,-2);

\begin{scope}[xshift=-2cm]
\draw[fill=gray!98] (0,-4) circle (7pt);
\draw[fill=gray!98] (0,4) circle (7pt);
\draw[fill=gray!98] (-4,0) circle (7pt);
\draw[fill=gray!98] (4,0) circle (7pt);
\draw[fill=gray!98] (2.8,2.8) circle (7pt);
\draw[fill=gray!98] (2.8,-2.8) circle (7pt);
\draw[fill=gray!98] (-2.8,2.8) circle (7pt);
\draw[fill=gray!98] (-2.8,-2.8) circle (7pt);
\draw[-latex] (0,4) -- (-2.8,2.8);
\draw[-latex] (-2.8,2.8) -- (-4,0);
\draw[-latex] (2.8,2.8) -- (0,4);
\draw[-latex] (4,0) -- (2.8,2.8);
\draw[latex-] (4,0) -- (2.8,-2.8);
\draw[-latex] (0,-4) -- (2.8,-2.8);
\draw[latex-] (0,-4) -- (-2.8,-2.8);
\draw[-latex] (-4,0) -- (-2.8,-2.8);
\end{scope}

\begin{scope}[xshift=10cm]
\draw[fill=gray!98] (0,-4) circle (7pt);
\draw[fill=gray!98] (0,4) circle (7pt);
\draw[fill=gray!98] (-4,0) circle (7pt);
\draw[fill=gray!98] (4,0) circle (7pt);
\draw[fill=gray!98] (2.8,2.8) circle (7pt);
\draw[fill=gray!98] (2.8,-2.8) circle (7pt);
\draw[fill=gray!98] (-2.8,2.8) circle (7pt);
\draw[fill=gray!98] (-2.8,-2.8) circle (7pt);
\draw[-latex] (0,4) -- (-2.8,2.8);
\draw[-latex] (-2.8,2.8) -- (-4,0);
\draw[-latex] (2.8,2.8) -- (0,4);
\draw[-latex] (4,0) -- (2.8,2.8);
\draw[latex-] (4,0) -- (2.8,-2.8);
\draw[-latex] (0,-4) -- (2.8,-2.8);
\draw[latex-] (0,-4) -- (-2.8,-2.8);
\draw[-latex] (-4,0) -- (-2.8,-2.8);
\end{scope}
\end{scope}

\begin{scope}[yshift=-24cm]
\begin{scope}[xshift=-2cm]
\draw[fill=gray!98] (0,-4) circle (7pt);
\draw[fill=gray!98] (0,4) circle (7pt);
\draw[fill=gray!98] (-4,0) circle (7pt);
\draw[fill=gray!98] (4,0) circle (7pt);
\draw[fill=gray!98] (2.8,2.8) circle (7pt);
\draw[fill=gray!98] (2.8,-2.8) circle (7pt);
\draw[fill=gray!98] (-2.8,2.8) circle (7pt);
\draw[fill=gray!98] (-2.8,-2.8) circle (7pt);
\draw[-latex] (0,4) -- (-2.8,2.8);
\draw[-latex] (-2.8,2.8) -- (-4,0);
\draw[-latex] (2.8,2.8) -- (0,4);
\draw[-latex] (4,0) -- (2.8,2.8);
\draw[latex-] (4,0) -- (2.8,-2.8);
\draw[-latex] (0,-4) -- (2.8,-2.8);
\draw[latex-] (0,-4) -- (-2.8,-2.8);
\draw[-latex] (-4,0) -- (-2.8,-2.8);
\end{scope}

\begin{scope}[xshift=10cm]
\draw[fill=gray!98] (0,-4) circle (7pt);
\draw[fill=gray!98] (0,4) circle (7pt);
\draw[fill=gray!98] (-4,0) circle (7pt);
\draw[fill=gray!98] (4,0) circle (7pt);
\draw[fill=gray!98] (2.8,2.8) circle (7pt);
\draw[fill=gray!98] (2.8,-2.8) circle (7pt);
\draw[fill=gray!98] (-2.8,2.8) circle (7pt);
\draw[fill=gray!98] (-2.8,-2.8) circle (7pt);
\draw[-latex] (0,4) -- (-2.8,2.8);
\draw[-latex] (-2.8,2.8) -- (-4,0);
\draw[-latex] (2.8,2.8) -- (0,4);
\draw[-latex] (4,0) -- (2.8,2.8);
\draw[latex-] (4,0) -- (2.8,-2.8);
\draw[-latex] (0,-4) -- (2.8,-2.8);
\draw[latex-] (0,-4) -- (-2.8,-2.8);
\draw[-latex] (-4,0) -- (-2.8,-2.8);
\end{scope}
\end{scope}
\end{tikzpicture}
\end{center}
\caption{In this representation, the two periodic orbits of the 
systems are not distinguished in the first graph.\label{figure.example}}
\end{figure}
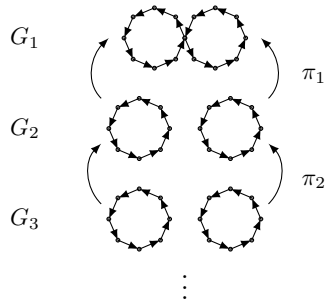
	
	\subsection{Attracting orbits\label{section.attracting.orbits}}

	\begin{definition}\label{definition.attracting.orbit}
		We say that a subset $u \subset X$ is \textbf{stable} (by $f$) when $f(u)\subset u$.
		A finite orbit $p$ of the system $(X,f)$ is called \textbf{attracting} whenever there exists a stable clopen set $u$ containing $p$ such that 
		\[\displaystyle{\bigcap_{n \ge 0} f^n (u)} = p.\] We also say that $u$ is \textbf{attracted} by $p$ or that $u$ is an \textbf{attracting neighborhood} of $p$.
	\end{definition}
	
	Note that if $p$ is attracting, then for every open set $v$ containing $p$ there exists 
	some integer $n \ge 1$ such that $f^n(u)\subset v$. It is also straightforward that 
	$f^n(u)$ is stable for every $n \ge 1$. 
	In the following, whenever considering a finite orbit $p$, we will 
	denote by $|p|$ its cardinality. The following is straightforward:

	\begin{lemma}\label{lemma.disjointness}
	Let us consider two distinct (and thus disjoint) finite orbits $p$ and $p'$ 
	and two stable clopen sets $u,u'$ which are respectively attracted by $p,p'$.
	The sets $u$ and $u'$ are disjoint.
	\end{lemma}

\begin{notation}
	In the present and following sections we consider a zero-dimensional 
	dynamical system $(X,f)$ whose finite orbits are all attracting. We also assume that $X$ is a Cantor set. For simplicity 
	we will designate such a system as \textbf{purely attracting} one.
\end{notation}

	\begin{lemma}\label{lemma.all.attracting}
		For all $n \ge 1$, a purely attracting system $(X,f)$ has only finitely many periodic orbits of period $n$.
	\end{lemma}
	
	\begin{proof}
		Since $X$ is compact, any infinite sequence of 
		periodic orbits of period $n$ has a subsequence 
		which converges, relatively to Hausdorff distance, to a finite orbit $p$ whose period is a divisor of $n$. Since this finite orbit is attracting, this is not possible.
	\end{proof}
	
	\begin{remark}
		Lemma~\ref{lemma.all.attracting} can be seen as a consequence of
		the statement \cite[Lemma~12]{ciesielski} that locally radially shrinking systems have a finite number of finite orbits of length $n$ for all integers $n \ge 1$. 
	\end{remark}
	
	\subsection{Supercyclical partitions\label{section.supercyclical.partitions}}

	\begin{definition}\label{definition.supercyclical}
	
	A \textbf{supercyclical partition} of 
	the system $(X,f)$ is a pair $(\mathcal{U},n)$, where $\mathcal{U}$ is 
	a finite clopen partition of $X$ and $n \ge 1$ is an integer, 
	such that there exists a sequence $(\leftidx{^p}{\mathcal{U}})_{|p|\le n}$ 
	of subsets of $\mathcal{U}$, where the indexes are periodic 
	orbits, such that for every 
	periodic orbit $p$ for which $|p| \le n$: 
	\begin{enumerate}
			\item[(S1)] $\mathcal{E}(\leftidx{^p}{\mathcal{U}})$ is stable, 
	contains $p$ and is attracted by it.
		\end{enumerate}	
	For a finite clopen partition $\mathcal{U}$, the largest integer $n$ (or infinity if there is no upper bound) such that $(\mathcal{U},n)$ is a supercyclical 
	partition is called the \textbf{supercyclical order} of $\mathcal{U}$, and is
	denoted by $o(\mathcal{U})$. By convention, if there is no $n \ge 1$ such that 
	$(\mathcal{U},n)$ is supercyclical, we set $o(\mathcal{U})=0$. Furthermore, 
	by extension, we say that a finite clopen partition is supercyclical when 
	$o(\mathcal{U}) \ge 1$.
	\end{definition}
	
	\begin{remark}
	This definition should not introduce any confusion, as we specify each time 
	if the supercyclical partition is attached with an integer $n$ or not.
	\end{remark}
	
	\begin{remark}
		Observe that if $(X,f)$ has only finitely periodic orbits whose periods are not larger than an integer $l$
		and  $l \le o(\mathcal{U})$ then  $o(\mathcal{U})=+\infty$.
	\end{remark}	
	
	Let us note that in Definition~\ref{definition.supercyclical}, as 
	a consequence of Lemma~\ref{lemma.disjointness}, for two different periodic 
	orbits $p$ and $p'$, the sets $\mathcal{E}(\leftidx{^p}{\mathcal{U}})$ and $\mathcal{E}(\leftidx{^{p'}}{\mathcal{U}})$ are disjoint. The following proposition 
	shows furthermore that there is a canonical choice for the sets $\leftidx{^p}{\mathcal{U}}$:

	\begin{proposition}\label{proposition.canonical}
	Let us consider $\mathcal{U}$ a supercyclical partition of the system $(X,f)$, 
	and $p$ a finite orbit such that $|p|\le o(\mathcal{U})$. The set of 
	subsets $U$ of $\mathcal{U}$ such that $\mathcal{E}(U)$ is attracted by $p$ 
	admits a maximum for the inclusion relation.
	\end{proposition} 
	
	\begin{proof}
	Indeed, it is sufficient to consider the union of all these sets $U$.
	\end{proof}

	\begin{notation}
	We will denote $\leftidx{_{*}^p}{\mathcal{U}}$ the maximum provided by Proposition~\ref{proposition.canonical}. The collection of the sets $\leftidx{_{*}^p}{\mathcal{U}}$ for $|p| \le n$ makes $(\mathcal{U},n)$ a  
	supercyclical partition. We will call \textbf{attracted part} of $(\mathcal{U},n)$ 
	the union of the sets $\leftidx{_{*}^p}{\mathcal{U}}$ for $|p| \le n$. The remainder of the 
	partition is called its \textbf{supercyclical part}. For simplicity we will also call 
	- without introducing ambiguity - attracted part and supercyclical 
	part the images by $\mathcal{E}$ of these respective sets. For the same reason, 
	for a finite clopen partition $\mathcal{U}$, we will call supercyclical part and attracted part of $\mathcal{U}$ respectively the supercyclical part and attracted 
	part of $(\mathcal{U},o(\mathcal{U}))$.
	\end{notation} 

\begin{remark}
The attachement of an integer $n$ to the definition of supercyclical partition 
will be used in the following in order to specify what we consider in the context to be the attracted part of the supercyclical partition. The purpose of doing so is to 
prevent modifications on the supercyclical part to affect the attracted part.
\end{remark}

	\begin{remark}
	When $f$ is surjective, since $X$ is without isolated points, the supercyclical part 
	of every supercyclical partition $\mathcal{U}$ is non-empty. Moreover if $(X,f)$ is aperiodic 
	then every clopen partition $\mathcal{U}$ of $X$ is supercyclical, $o(\mathcal{U}) = +\infty$ and the attracted part $\mathcal{U}$ is empty.
	\end{remark}

	The following is straightforward: 
	
	\begin{lemma}\label{lemma.supercyclical.stability}
	Let us consider $\mathcal{U}$ a supercyclical partition for $(X,f)$  and 
	$\mathcal{V}$ another finite clopen partition such that $\mathcal{V} \prec \mathcal{U}$. 
	Then the partition $\mathcal{V}$ is supercyclical and its order is at least 
	$o(\mathcal{U})$. Furtheremore for all $p$ such that $|p| \le o(\mathcal{U})$, 
	we have that $\mathcal{E}(\leftidx{_*^p}{\mathcal{U}}) \subset \mathcal{E}(\leftidx{_*^p}{\mathcal{V}})$.
	\end{lemma} 	
	
	\begin{lemma} \label{lemma.augmenting.order} 
	Let us consider $\mathcal{U}$ a supercyclical partition for $(X,f)$ and assume that $o(\mathcal{U})<+\infty$.
	There exists another supercyclical partition 
	$\mathcal{V}$ of order at least $o(\mathcal{U})+1$ which refines $\mathcal{U}$.
	\end{lemma} 
	
	\begin{proof}
	Every finite orbit $p$ such that $|p|\ge o(\mathcal{U})+1$ is contained in the supercyclical part of $\mathcal{U}$. Since  $o(\mathcal{U})<+\infty$ there exists at least one such orbit $p$. Let $l \geq o(\mathcal{U})+1$ be the smallest possible
	period of these orbits. The supercyclical part of $\mathcal{U}$ is a clopen set, hence for every periodic orbit $p$ of period $l$ there exists a clopen attracted neighborhood $u_p$ of $p$ which is contained in the supercyclical part of 
	$\mathcal{U}$. 
	There exists $m>0$ 
	such that each of the sets $u_p$ is a union of elements of $\mathcal{U}_m^0$ (this sequence was fixed 
	at the beginning of Section~\ref{section.graph.coverings}), and the 
	same holds for the complement, in the supercyclical part of $\mathcal{U}$,
	of the union of these sets. We then 
	define $\mathcal{V}$ by collecting the following elements: the elements of $\mathcal{U}$ 
	contained in the attracted part of $\mathcal{U}$ and the elements of $\mathcal{U}_m^0$ contained in the 
	supercyclical part of $\mathcal{U}$. For every finite orbit $p$ with $|p| \le o(\mathcal{U})$, 
	we consider $\leftidx{^p}{\mathcal{V}} = \leftidx{_*^p}{\mathcal{U}}$ 
	and when $|p|=l$, 
	$\leftidx{^p}{\mathcal{V}}$ consists of the sets
	$u \in \mathcal{U}^0_m$ such that $u \subset u_p$. Since 
	these sets sum up to $u_p$ which is attracted to $p$, the sequence $(\leftidx{^p}{\mathcal{V}})_{|p|\le l}$
	makes the partition $\mathcal{V}$ supercyclical of order at least $o(\mathcal{U})+1$.
	\end{proof}

	Let us consider $\mathcal{U}$ a supercyclical partition for $(X,f)$.
	For all finite orbit $p$ such that $|p| \le o(\mathcal{U})$, 
	the subgraph of $G(\mathcal{U})$ which corresponds to $\leftidx{_*^p}{\mathcal{U}}$ 
does not contain edges from a vertex in this part 
	to another vertex outside of it since $\leftidx{_*^p}{\mathcal{U}}$ is a stable set (see Figure~\ref{figure.almost.acyclic} for an illustration). 
While $\mathcal{E}(\leftidx{_*^p}{\mathcal{U}})$ contains a unique periodic orbit, the graph associated with $\leftidx{_*^p}{\mathcal{U}}$ may contain other circuits than the one corresponding 
to the finite orbit, although they do not represent finite orbits but are 
traces of infinite orbits. We will thus refine the 
sets $\leftidx{_*^p}{\mathcal{U}}$, as partitions of the respective sets $\mathcal{E}(\leftidx{_*^p}{\mathcal{U}})$, so that the associated graph contains a unique circuit.

	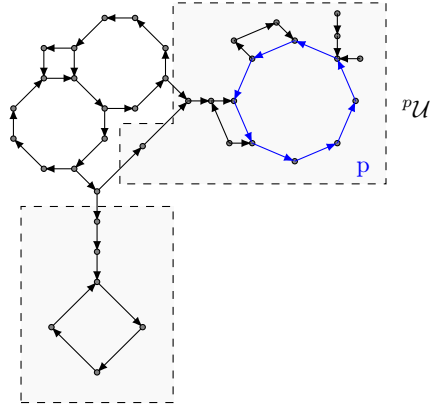
\begin{figure}[h!]
		\begin{center}
			\begin{tikzpicture}[scale=0.2]
			\draw[dashed,fill=gray!5]
			(-11.5,-5.5) -- (-11.5,-1.5) -- (-8,-1.5) -- (-8,6.5) -- (6,6.5) -- (6,-5.5) 
			-- (-11.5,-5.5);
			\draw[fill=gray!98] (0,-4) circle (5pt);
			\draw[fill=gray!98] (0,4) circle (5pt);
			\draw[fill=gray!98] (-4,0) circle (5pt);
			\draw[fill=gray!98] (4,0) circle (5pt);
			\draw[fill=gray!98] (2.8,2.8) circle (5pt);
			\draw[fill=gray!98] (2.8,-2.8) circle (5pt);
			\draw[fill=gray!98] (-2.8,2.8) circle (5pt);
			\draw[fill=gray!98] (-2.8,-2.8) circle (5pt);
			\draw[fill=gray!98] (-4.3,-2.8) circle (5pt);
			\draw[fill=gray!98] (-5.5,0) circle (5pt);
			\draw[fill=gray!98] (-7,0) circle (5pt);
			\draw[fill=gray!98] (4.3,2.8) circle (5pt);
			\draw[fill=gray!98] (2.8,4.3) circle (5pt);
			\draw[fill=gray!98] (2.8,5.8) circle (5pt);
			\draw[fill=gray!98] (-4,4) circle (5pt);
			\draw[fill=gray!98] (-1.2,5.2) circle (5pt);
			
			\draw[-latex] (-2.8,2.8) -- (-4,4);
			\draw[-latex] (-4,4) -- (-1.2,5.2);
			\draw[-latex] (-1.2,5.2) -- (0,4);

			\draw[-latex,color=blue] (0,4) -- (-2.8,2.8);
			\draw[-latex,color=blue] (-2.8,2.8) -- (-4,0);
			\draw[-latex,color=blue] (2.8,2.8) -- (0,4);
			\draw[-latex,color=blue] (4,0) -- (2.8,2.8);
			\draw[latex-,color=blue] (4,0) -- (2.8,-2.8);
			\draw[-latex,color=blue] (0,-4) -- (2.8,-2.8);
			\draw[latex-,color=blue] (0,-4) -- (-2.8,-2.8);
			\draw[-latex,color=blue] (-4,0) -- (-2.8,-2.8);
			\node[color=blue,scale=0.9] at (4.5,-4.5) {p};
			\node[scale=0.9] at (8,-0.5) {$\leftidx{^p}{\mathcal{U}}$};
			
			\draw[-latex] (-4.3,-2.8) -- (-5.5,0);
			\draw[-latex] (-4.3,-2.8) -- (-2.8,-2.8);
			\draw[-latex] (-5.5,0) -- (-4,0);
			\draw[latex-] (-5.5,0) -- (-7,0);
			
			\draw[dashed,fill=gray!5] (-18,-20) rectangle (-8,-7);
			
			\draw[-latex] (-10,-3) -- (-7,0);
			\draw[-latex] (-13,-6) -- (-10,-3);
			\draw[-latex] (-13,-6) -- (-13,-8);
			\draw[-latex] (-13,-8) -- (-13,-10);
			\draw[-latex] (-13,-10) -- (-13,-12);
			\draw[-latex] (-13,-12) -- (-10,-15);
			\draw[-latex] (-10,-15) -- (-13,-18);
			\draw[-latex] (-13,-18) -- (-16,-15);
			\draw[-latex] (-16,-15) -- (-13,-12);
			\draw[-latex] (-14.5,-4.5) -- (-13,-6);
			\draw[fill=gray!98] (-10,-3) circle (5pt);
			\draw[fill=gray!98] (-13,-6) circle (5pt);
			\draw[fill=gray!98] (-13,-8) circle (5pt);
			\draw[fill=gray!98] (-13,-10) circle (5pt);
			\draw[fill=gray!98] (-13,-12) circle (5pt);
			\draw[fill=gray!98] (-16,-15) circle (5pt);
			\draw[fill=gray!98] (-13,-18) circle (5pt);
			\draw[fill=gray!98] (-10,-15) circle (5pt);
			
			\draw[-latex] (-8.5,1.5) -- (-7,0);
			\draw[-latex] (-8.5,1.5) -- (-8.5,3.5);
			\draw[-latex] (-8.5,3.5) -- (-10.5,5.5);
			\draw[-latex] (-10.5,5.5) -- (-12.5,5.5);
			\draw[-latex] (-12.5,5.5) -- (-14.5,3.5);
			\draw[-latex] (-14.5,3.5) -- (-16.5,3.5);
			\draw[-latex] (-16.5,3.5) -- (-16.5,1.5);
			\draw[-latex] (-16.5,1.5) -- (-14.5,1.5);
			\draw[-latex] (-14.5,3.5) -- (-14.5,1.5);
			\draw[-latex] (-14.5,1.5) -- (-12.5,-0.5);
			\draw[-latex] (-12.5,-0.5) -- (-10.5,-0.5);
			\draw[-latex] (-10.5,-0.5) -- (-8.5,1.5);
			\draw[-latex] (-18.5,-0.5) -- (-16.5,1.5);
			\draw[-latex] (-18.5,-2.5) -- (-18.5,-0.5);
			\draw[-latex] (-16.5,-4.5) -- (-18.5,-2.5);
			\draw[-latex] (-14.5,-4.5) -- (-16.5,-4.5);
			\draw[-latex] (-12.5,-2.5) -- (-14.5,-4.5);
			\draw[-latex] (-12.5,-0.5) -- (-12.5,-2.5);
			
			\draw[fill=gray!98] (-12.5,-2.5) circle (5pt);
			\draw[fill=gray!98] (-14.5,-4.5) circle (5pt);
			\draw[fill=gray!98] (-16.5,-4.5) circle (5pt);
			\draw[fill=gray!98] (-18.5,-2.5) circle (5pt);
			\draw[fill=gray!98] (-18.5,-0.5) circle (5pt);
			\draw[fill=gray!98] (-8.5,3.5) circle (5pt);
			\draw[fill=gray!98] (-10.5,5.5) circle (5pt);
			\draw[fill=gray!98] (-12.5,5.5) circle (5pt);
			\draw[fill=gray!98] (-14.5,3.5) circle (5pt);
			\draw[fill=gray!98] (-14.5,1.5) circle (5pt);
			\draw[fill=gray!98] (-12.5,-0.5) circle (5pt);
			\draw[fill=gray!98] (-10.5,-0.5) circle (5pt);
			\draw[fill=gray!98] (-8.5,1.5) circle (5pt);
			\draw[fill=gray!98] (-16.5,1.5) circle (5pt);
			\draw[fill=gray!98] (-16.5,3.5) circle (5pt);
			
			\draw[-latex] (2.8,4.3) -- (2.8,2.8);
			\draw[-latex] (4.3,2.8) -- (2.8,2.8);
			\draw[latex-] (2.8,4.3) -- (2.8,5.8);
			\end{tikzpicture}
		\end{center}
		\caption{Illustration of the graph $G(\mathcal{U})$ for a supercyclical 
		partition $\mathcal{U}$ of $(X,f)$. The dashed regions
		correspond to the attracted part of the partition, the remainder
		corresponds to the supercyclical part. }\label{figure.almost.acyclic}
	\end{figure}

	\subsection{Refinement of the attracted part\label{section.attracted.operations}}
	
	In the following, we will need to have the following additional properties on the graph 
	of each constructed supercyclical partition: 
	\begin{enumerate}
	\item[(S2)] 
	each of the subgraphs of the attracted part corresponding to some $\leftidx{_*^p}{\mathcal{U}}$
	contains a unique circuit.
	\item[(S3)] 
	in the attracted part there 
	is not divergent vertex (a vertex with at least two outgoing 
edges).
	\end{enumerate}
	 We would like to refine supercyclical partitions in order to remove 
	these vertices, as illustrated on Figure~\ref{figure.suppression}. As a byproduct, we will also remove ''artificial'' circuits.
	This is the purpose 
	of the present section.
	
	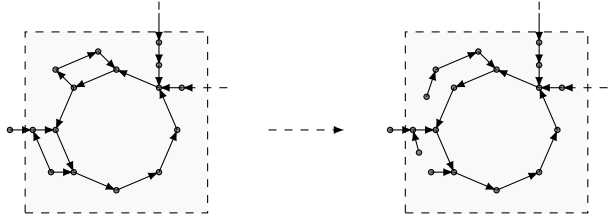
\begin{figure}[h!]
		\begin{center}
			\begin{tikzpicture}[scale=0.2]
			\draw[dashed,fill=gray!5] (-6,-5.5) rectangle (6,6.5);
			\draw[fill=gray!98] (0,-4) circle (5pt);
			\draw[fill=gray!98] (0,4) circle (5pt);
			\draw[fill=gray!98] (-4,0) circle (5pt);
			\draw[fill=gray!98] (4,0) circle (5pt);
			\draw[fill=gray!98] (2.8,2.8) circle (5pt);
			\draw[fill=gray!98] (2.8,-2.8) circle (5pt);
			\draw[fill=gray!98] (-2.8,2.8) circle (5pt);
			\draw[fill=gray!98] (-2.8,-2.8) circle (5pt);
			\draw[fill=gray!98] (-4.3,-2.8) circle (5pt);
			\draw[fill=gray!98] (-5.5,0) circle (5pt);
			\draw[fill=gray!98] (-7,0) circle (5pt);
			\draw[fill=gray!98] (4.3,2.8) circle (5pt);
			\draw[fill=gray!98] (2.8,4.3) circle (5pt);
			\draw[fill=gray!98] (2.8,5.8) circle (5pt);
			\draw[-latex] (2.8,7.75) -- (2.8,5.8);
			\draw[dashed] (2.8,8.5) -- (2.8,7.75);
			\draw[dashed,-latex] (7.3,2.8) -- (4.3,2.8);
			\draw[-latex] (0,4) -- (-2.8,2.8);
			\draw[-latex] (-2.8,2.8) -- (-4,0);
			\draw[-latex] (2.8,2.8) -- (0,4);
			\draw[-latex] (4,0) -- (2.8,2.8);
			\draw[latex-] (4,0) -- (2.8,-2.8);
			\draw[-latex] (0,-4) -- (2.8,-2.8);
			\draw[latex-] (0,-4) -- (-2.8,-2.8);
			\draw[-latex] (-4,0) -- (-2.8,-2.8);
			
			\draw[fill=gray!98] (-4,4) circle (5pt);
			\draw[fill=gray!98] (-1.2,5.2) circle (5pt);
			
			\draw[-latex] (-2.8,2.8) -- (-4,4);
			\draw[-latex] (-4,4) -- (-1.2,5.2);
			\draw[-latex] (-1.2,5.2) -- (0,4);

			\draw[-latex] (-4.3,-2.8) -- (-5.5,0);
			\draw[-latex] (-4.3,-2.8) -- (-2.8,-2.8);
			\draw[-latex] (-5.5,0) -- (-4,0);
			\draw[latex-] (-5.5,0) -- (-7,0);

			\draw[-latex] (2.8,4.3) -- (2.8,2.8);
			\draw[-latex] (4.3,2.8) -- (2.8,2.8);
			\draw[latex-] (2.8,4.3) -- (2.8,5.8);

			\draw[dashed,-latex] (10,0) -- (15,0);			
			
			\begin{scope}[xshift=25cm]
			
			\draw[dashed,fill=gray!5] (-6,-5.5) rectangle (6,6.5);
			\draw[fill=gray!98] (0,-4) circle (5pt);
			\draw[fill=gray!98] (0,4) circle (5pt);
			\draw[fill=gray!98] (-4,0) circle (5pt);
			\draw[fill=gray!98] (4,0) circle (5pt);
			\draw[fill=gray!98] (2.8,2.8) circle (5pt);
			\draw[fill=gray!98] (2.8,-2.8) circle (5pt);
			\draw[fill=gray!98] (-2.8,2.8) circle (5pt);
			\draw[fill=gray!98] (-2.8,-2.8) circle (5pt);
			\draw[fill=gray!98] (-4.3,-2.8) circle (5pt);
			\draw[fill=gray!98] (-5.5,0) circle (5pt);
			\draw[fill=gray!98] (-7,0) circle (5pt);
			\draw[fill=gray!98] (4.3,2.8) circle (5pt);
			\draw[fill=gray!98] (2.8,4.3) circle (5pt);
			\draw[fill=gray!98] (2.8,5.8) circle (5pt);
			\draw[-latex] (2.8,7.75) -- (2.8,5.8);
			\draw[dashed] (2.8,8.5) -- (2.8,7.75);
			\draw[dashed,-latex] (7.3,2.8) -- (4.3,2.8);
			\draw[-latex] (0,4) -- (-2.8,2.8);
			\draw[-latex] (-2.8,2.8) -- (-4,0);
			\draw[-latex] (2.8,2.8) -- (0,4);
			\draw[-latex] (4,0) -- (2.8,2.8);
			\draw[latex-] (4,0) -- (2.8,-2.8);
			\draw[-latex] (0,-4) -- (2.8,-2.8);
			\draw[latex-] (0,-4) -- (-2.8,-2.8);
			\draw[-latex] (-4,0) -- (-2.8,-2.8);
			
			\draw[-latex] (-4.3,-2.8) -- (-2.8,-2.8);
			\draw[-latex] (-5.5,0) -- (-4,0);
			\draw[latex-] (-5.5,0) -- (-7,0);
			
			\draw[-latex] (-5.1,-1.5) -- (-5.5,0);
			
			\draw[fill=gray!98] (-4,4) circle (5pt);
			\draw[fill=gray!98] (-1.2,5.2) circle (5pt);
			\draw[fill=gray!98] (-4.6,2.2) circle (5pt);
			\draw[fill=gray!98] (-5.1,-1.5) circle (5pt);
			
			\draw[-latex] (-4.6,2.2) -- (-4,4);
			\draw[-latex] (-4,4) -- (-1.2,5.2);
			\draw[-latex] (-1.2,5.2) -- (0,4);
			
			\draw[-latex] (2.8,4.3) -- (2.8,2.8);
			\draw[-latex] (4.3,2.8) -- (2.8,2.8);
			\draw[latex-] (2.8,4.3) -- (2.8,5.8);
			\end{scope}
			\end{tikzpicture}
		\end{center}
		\caption{Removing the divergent vertices in the attracted part.}\label{figure.suppression}
	\end{figure} \bigskip
	
	Let us consider $\mathcal{U}$ a supercyclical partition. In this section we define
	a sequence $(\kappa_n (\mathcal{U}))_{n \ge 0}$ of supercyclical partitions 
	such that for all $n$, $\kappa_n(\mathcal{U})$ has the same order as $\mathcal{U}$ 
	and such that:
	\begin{enumerate}[(i)]
	\item for all $n \ge 0$, $\kappa_{n+1}(\mathcal{U}) 
	\prec \kappa_{n}(\mathcal{U})$;
	\item $\kappa_0(\mathcal{U}) = \mathcal{U}$; 
	\item for all $n \ge 1$ the supercyclical part of 
	$\kappa_n(\mathcal{U})$ is identical to the one of $\mathcal{U}$. 
	\end{enumerate}
We then prove 
some properties of these partitions related to the graph representation. 
	
	Let us recall that we fixed  a sequence of 
	clopen partition $(\mathcal{U}_n^0)_{n \ge 1}$ which satisfies the condition \eqref{eq:star}. 
	Let us define the sequence ($\kappa_n(\mathcal{U}))_{n \ge 1}$ 
	by separately defining the supercyclical part and the attracted 
	part. For all $n \ge 1$, the partition $\kappa_n(\mathcal{U})$ has the following elements: 
	the ones of $\mathcal{U}$ that are in its supercyclical part; and the elements of 
	some sets $\mathcal{H}_n(p)$ defined below, for $|p| \le o(\mathcal{U})$, such that 
	for all $p$, the elements of $\mathcal{H}_n(p)$ form a partition of 
	$\mathcal{E}(\leftidx{_*^p}{\mathcal{U}})$. \bigskip 
	
	Let us define the sequence $(\mathcal{H}_n(p))_{n \ge 1}$ recursively 
	for all $p$ such that $|p| \le o(\mathcal{U})$. The principle of the definition 
	is to ``\textit{track back}'' how the points in 
	$\mathcal{E}(\leftidx{_*^p}{\mathcal{U}})$ approach the orbit $p$. 
	Let us fix one periodic orbit $p=\{p_1,\ldots,p_{k}\}$, $k \le o(\mathcal{U})$.
	Assume standard ordering on the orbit, that is $p_l=f^{l-1}(p_1)$ for all $l \le k$. Let $n_0$ be the smallest integer $s$ such that: 
	\begin{enumerate}
	\item there exists a partition 
	of $\mathcal{E}(\leftidx{_*^p}{\mathcal{U}})$ with elements of $\mathcal{U}^0_s$;
	\item this partition refines 
	$\leftidx{_*^p}{\mathcal{U}}$ as a partition of $\mathcal{E}(\leftidx{_*^p}{\mathcal{U}})$;
	\end{enumerate}
	
	Before going further, let us 
	prove that there exist $\hat{u}_1$, $\hat{u}_2$, ... , $\hat{u}_k$ disjoint 
	clopen sets such that for all $l \le k$,
	$\hat{u}_l$ contains $p_l$, and  
	$f(\hat{u}_l) \subset \hat{u}_{l+1}$, 
	where for convenience we set $\hat{u}_{k+1} := \hat{u}_1$.
	Since $p_1$ is periodic orbit, there is an open set $u\ni p_1$ such that $f^i(\overline{u})\cap f^j(\overline{u})=\emptyset$ for $i\neq j$.
	By the well-known property of attracting sets (e.g. see Proposition V.15 in \cite{BC}), there is an open set $w \ni p_1$ such that $\overline{w}\subset u$
	and $f^k(\overline{w})\subset w$. Taking a finite cover of $\overline{w}$ by clopen sets of sufficiently small diameter, we obtain a
	clopen set $\hat{u}_1\ni p_1$ such that $\overline{w}\subset \hat{u}_1\subset u$ and $f^k(\hat{u}_1)\subset w \subset \hat{u}_1$.
	Since $f(\hat{u}_1)$ is closed and $\hat{u}_1$ is open, we can find again a sufficiently small clopen neighborhood $\hat{u}_{2} \supset f(\hat{u}_1)$
	such that $f^{k-1} (\hat{u}_2) \subset \hat{u}_1$. In the same way we 
	find clopen sets $\hat{u}_3, \ldots, \hat{u}_k$. By construction we have 
	that $f(\hat{u}_l) \subset \hat{u}_{l+1}$ for all $l$. Furthermore the 
	disjointness of these sets derives from  $f^i(\overline{u})\cap f^j(\overline{u})=\emptyset$ for $i\neq j$.
	\bigskip 
	
	For all $n \ge 0$ we denote by $\mathcal{H}_n (p)$ the partition of 
	$\mathcal{E}(\leftidx{_*^p}{\mathcal{U}})$ defined as:
	\[\mathcal{H}_n (p) = \bigcup_{m \ge 0} S^{(n)}_m(p),\]
	where the sequence $(S^{(n)}_m(p))_{m \ge 0}$ is constructed as follows: 
	
	\begin{enumerate}[(i)]
		\item consider $w^n$ some clopen set contained in $\displaystyle{\bigcup_{l \le k}} \hat{u}_l$ and which consists in some union of elements 
		of $\mathcal{U}_{n+n_0}$ and minimal such that $w^n$ is stable under $f$;
		\item define $S^{(n)}_0(p)= \{w^n_1, ... , w^n_k\}$, where $w^n_l = w^n \cap \hat{u}_l$;
		\item $S^{(n)}_1(p) = \{ (v \cap f^{-1}(w_l^n))\backslash w^n \ : \ 1 \le l \le k, \ v \in \mathcal{U}^0_{n+n_0}, \ v \subset \mathcal{E}(\leftidx{^p}{\mathcal{U}})\}$;
		\item for all $m \ge 1$, $S^{(n)}_{m+1}(p) = \{v \cap f^{-1}(w) \ : \ w \in S^{(n)}_m(p), v \in \mathcal{U}^0_{n+n_0}, \ v \subset \mathcal{E}(\leftidx{^p}{\mathcal{U}})\}$.
	\end{enumerate}
	
	In words the points in the union of elements of $S^{(n)}_m(p)$ are the ones 
	that arrive in some $w_l^n$ after $m$ iterations of $f$. This idea is similar to 
	the one used in the construction of Gambaudo-Martens 
	representation~\cite{Gambaudo-Martens}, although resulting graphs are quite different. \bigskip
	
	It is straightforward to see that the elements of $\mathcal{H}_n(p)$ are disjoint and edges starting in $S^{(n)}_{m+1}(p)$
	have to end in $S^{(n)}_m(p)$ for all $m\geq 0$. 
	Moreover, for each $x \in \mathcal{E}(\leftidx{_*^p}{\mathcal{U}})$, there exists 
	some $m$ such that $f^m (x) \in w^n$. As a consequence $\mathcal{H}_n(p)$ covers 
	$\mathcal{E}(\leftidx{_*^p}{\mathcal{U}})$, which is compact, meaning that 
	only finitely many sets $S^{(n)}_m(p)$ are not empty. Thus $\kappa_n(\mathcal{U})$ 
	is indeed a finite partition of $X$.
	Schematic illustration of next Lemma can be found on the right part of Figure~\ref{figure.suppression}.
It shows that our construction indeed leads to conditions (S2) and (S3).		
	
\begin{remark}
In order to obtain (S2) and (S3) and get rid of divergent vertices, it is enough to use only $\kappa_1(\mathcal{U})$. However we present the construction with $\kappa_n$
since it may be of independent interest for further research, since it allows to refine attracted part without changing supercyclic part.
\end{remark}
	
	\begin{lemma}\label{lemma.graph.supercyclical}
	Consider a supercyclical partiton $\mathcal{U}$. For all $n \ge 1$, the restriction 
	of the graph $G(\kappa_n(\mathcal{U}))$ to any $\leftidx{_*^p}{\kappa_n(\mathcal{U})}$ 
	with $|p| \le o(\mathcal{U})$ consists in a unique circuit, together with a finite 
	number of paths whose intersection with the circuit is reduced to the 
	endpoint of this path and if two such paths coincide at some point, 
	they coincide until their endpoint.
	\end{lemma}
		
	\begin{proof}
	Let us fix a finite orbit $p$ such that $|p| \le o(\mathcal{U})$. 
	It is straightforward that the elements of $S^{(n)}_0(p)$ form a circuit in 
	the graph $G(\kappa_n(\mathcal{U}))$. Let us consider some $u \in \mathcal{H}_n(p)$ 
	which is not in this circuit. By construction there exists $m \ge 1$ such that 
	$u \in S^{(n)}_m(p)$. There is a sequence of vertices $(u_l)_{0 \le l \le m}$
	such that for all $l$, $u_l \in S^{(n)}_{m-l}$ and for $l \le m-1$, 
	$(u_l,u_{l+1})$ is an edge of the graph $G(\kappa_n(\mathcal{U}))$. 
	By definition, each element of some $S^{(n)}_m$, $m \ge 1$ 
	is connected to a unique element of $S^{(n)}_{m-1}$. This implies the lemma.
	\end{proof}

	\subsection{Markers compatible with supercyclical partitions\label{section.krieger}}
	
In the following, we will use an adaptation of Krieger's notion of markers in order to mark partitions.
	
\begin{definition}
Let us consider $\mathcal{U}$ a supercyclical partition for $(X,f)$ and  $1\le n \le o(\mathcal{U})$, and let us denote by
$\mathcal{S}$ the supercyclical part of $\mathcal{U}$. A \textbf{$(n,t,N)$-marker} for $\mathcal{U}$ 
is clopen set $F \subset \E(\mathcal{S})$ such that
\begin{enumerate}
\item the set $F$ is $(n+1)$-separated, meaning that the sets $F, f^{-1}(F),\ldots  , f^{-n}(F)$ are pairwise disjoint;
\item the sets $F, f^{-1}(F), \ldots , f^{-N}(F)$ cover $f^{-t}(\mathcal{\E(S)})$.
\end{enumerate}
\end{definition}
	
	\begin{theorem}\label{theorem.krieger}
	Fix $n \ge 1$  and let us consider $\mathcal{U}$ a supercyclical partition of order at least $n$ for 
	$(X,f)$. There exists some integer $m \ge 1$ such that $\mathcal{U}$ admits 
	an $(n,(n+1)m+n,2n+1)$-marker.
	\end{theorem}
		
	\begin{proof}
	The proof is very similar to the one presented in~\cite{Downarowicz} for 
	aperiodic Cantor systems. We adapt it here to a purely attracting system, 
	with appropriate changes to our slightly more general setting. 
		
	For each $x$ in the supercyclical part of $\mathcal{U}$ there exists 
	$u_x$ clopen neighborhood of $x$ which is included in the 
	supercyclical part and which is $(n+1)$-separated. Since the attracted part of $\mathcal{U}$
	is stable by $f$, for all $k \le n$, $f^{-k}(u_x)$ is included in the supercyclical part.
	The clopen sets $u_x$ for $x$ in the supercyclical part form an open cover of this part, 
	which is a compact set. As a consequence there is a finite subcover of it with these 
	clopen sets. Let us denote $u_j$, $j \in \{1,\ldots,m\}$ the elements of this 
	subcover. The sets $u'_j = f^{-(n+1)m}(u_j)$ are all $(n+1)$-separated, and 
	they cover $f^{-(n+1)m} (\mathcal{\E(\mathcal{S})})$. \bigskip 
	
	Let us define $F_1 = u'_1$ and recursively: 
	\[F_{j+1} := F_j \bigcup \left(u'_{j+1} \backslash \bigcup_{-(n+1) < i < (n+1)} f^{-i}(F_j)\right).\]
	Finally set $F = F_m$. Let us prove that $F$ is a $(n,(n+1)m+n,2n+1)$-marker: 
	\begin{enumerate}
	\item \textbf{The set $F$ is $(n+1)$-separated.} Let us prove inductively that each $F_j$ is 
	$(n+1)$-separated. This is immediate for $F_1$. Assume the property is verified for $F_{j-1}$ and 
	that, ad absurdum, that $F_j$ is not $(n+1)$-separated. As a direct consequence, there exist 
	two $i',i''$ such that $0 \le i' < i'' < n+1$ such that 
	$f^{-i'}(F_j) \bigcap f^{-i''}(F_j) \neq \emptyset$. Thus we have 
	$F_j \bigcap f^{-i}(F_j) \neq \emptyset$ for some $0 < i < n+1$, and therefore there exists $x$ such that 
	$x \in F_j$ and $f^i(x) \in F_j$. We have, by definition of $F_j$, that $F_j \subset F_{j-1}
	\cup u'_{j}$. Since both of these sets are $(n+1)$-separated, the two points $x$ and 
	$f^i(x)$ belong to
	different of them. The point which belongs to $u'_j$ also belongs to 
	\[\displaystyle{\bigcup_{-(n+1) < i < (n+1)}} f^{-i}(F_{j-1}),\]
	substracted from $u'_j$ when defining $F_j$. So that point does not belong to $F_j$, 
	a contradiction.
	\item \textbf{Every point of $f^{-(n+1)m-n}(\E(\mathcal{S}))$ visits $F$ at least once in $2n+1$ iterations.} 
	Consider a point $x$ in the set $f^{-(n+1)m-n}(\mathcal{S})$. Then $f^{n}(x)$ belongs 
	to some $u'_{j}$. If $f^{n} (x) \in F_j$, then $f^{n}(x) \in F$.
	The only way that $f^{n} (x)$ may not belong to $F_j$ is if it belongs to 
	\[\displaystyle{\bigcup_{-(n+1) < i < (n+1)}} f^{-i}(F_{j-1}).\]
	Then we obtain that $x \in f^{-n-i} (F)$ with $-n-1 < i < n+1$. This concludes the proof.
	\end{enumerate} 
	\end{proof}

	\subsection{Well marked partitions\label{section.marking}}
	
	In this section we define marked partitions, as well as constraints on the marker system 
	that we will need to be satisfied when constructing a suitable sequence of partitions for the 
	system $(X,f)$. 	
	
	\begin{definition}
	A \textbf{marked partition} is a pair $(\mathcal{U},\tau,\chi)$ where $(\mathcal{U},\tau)$ is 
	a supercyclical partition for $(X,f)$, and 
	$\chi : \mathcal{U} \rightarrow \{\uparrow,*,0,\downarrow\}$. 
	An element $u$ of $\mathcal{U}$ such that $\chi(u)=*$ is called a \textbf{marker}. 
	When $\chi(u)=\uparrow$, $u$ is called a \textbf{potential}.
	\end{definition}
	
	\begin{definition}
	We say that a marked partition $(\mathcal{U},\tau,\chi)$ is \textbf{well marked} when the 
	function $\chi$ has the following 
	properties: 
	\begin{enumerate}
	\item For all $u$ in the attracted part of $(\mathcal{U},\tau)$, $\chi(u)=0$; for all 
	other $u$, $\chi(u) \neq 0$.
	\item The circuits whose vertices are all contained in the supercyclical part of $(\mathcal{U},\tau)$ all contain at least one marker and one potential.
	\end{enumerate}
	\end{definition}
	
	The following is straightforward: 
	
	\begin{lemma}\label{lemma.transmission.well.marked}
	Let us consider a well marked supercyclical partition 
	$(\mathcal{U},\tau,\chi)$ and another finite clopen partition 
	$(\mathcal{V},\tau')$ where $\mathcal{V}$ refines $\mathcal{U}$ and $\tau' \ge \tau$. Let us denote $\chi' : \mathcal{V} \rightarrow
	\{\uparrow,\downarrow,*,0\}$ whose value is $0$ on the attracted part of $(\mathcal{V},\tau')$ 
	and such that for all $v \in \mathcal{V}$ in the supercyclical part of $(\mathcal{V},\tau')$, 
	$\chi(\pi_{\mathcal{U}}^{\mathcal{V}}(v)) = \chi'(v)$. Then $(\mathcal{V},\tau',\chi')$ 
	is well-marked.
	\end{lemma}

	\begin{definition}\label{def:markersymbols}
	Let us consider two well-marked partitions 
	$(\mathcal{U},\tau,\chi)$ and $(\mathcal{V},\tau',\chi')$ such that $\mathcal{V} \prec \mathcal{U}$.
	We say that $(\mathcal{V},\tau',\chi')$ is \textbf{well-marked relatively to} $(\mathcal{U},\tau,\chi)$ 
	when for all $u$ in the 
	supercyclical part of $(\mathcal{V},\tau')$, 
	the word $\chi'(u) \chi\left(\pi_{\mathcal{U}}^{\mathcal{V}}(u)\right)$ 
	is in $\{*\uparrow, \uparrow \uparrow, \downarrow \uparrow, \downarrow *, \downarrow \downarrow\}$. 
	\end{definition}

In particular, Definition~\ref{def:markersymbols} means that for a position in the supercyclical part of $(\mathcal{V},\tau)$, if it is mapped by $\pi_{\mathcal{U}}^{\mathcal{V}}$ to a potential then 
it is a marker, a potential or is marked with $\downarrow$. Otherwise it is marked with 
$\downarrow$.

	\begin{notation}
	For a finite clopen supercyclical partition $\mathcal{U}$, we denote by 
	$\eta(\mathcal{U})$ 
	the maximal length of a circuit in $G(\mathcal{U})$ whose vertices are 
	all contained in the supercyclical part of $\mathcal{U}$. 
	\end{notation}
	
	\begin{lemma}\label{lemma.well.marked}
	Let us consider a well-marked partition $(\mathcal{U},\tau,\chi)$ and 
	some integer $n$ larger or equal to $1$. There exists a
	well marked partition $(\mathcal{V},\tau',\chi')$, where $\mathcal{V}$ is of order 
	at least $\max(2 \eta(\mathcal{U}),o(\mathcal{U})+1)$ such that $\mathcal{V} \prec \mathcal{U}$, $\mathcal{V} \prec \mathcal{U}^0_n$, $\tau' \ge \tau+1$,
	and $(\mathcal{V},\tau',\chi')$ is well-marked relatively to $(\mathcal{U},\tau,\chi)$.
	\end{lemma}
	
	\begin{proof}
\textbf{(1)} \textbf{Setup:}
	We first consider $\mathcal{V}_0$ to be the refinement $\mathcal{U} \vee \mathcal{U}^0_{n+t}$, for $t$ large enough so that $o(\mathcal{V}_0) \ge o(\mathcal{U})+1$. If $ o(\mathcal{V}_0) < 2 \eta(\mathcal{U})$, 
	we use Lemma~\ref{lemma.augmenting.order} in order to refine this partition into 
	another supercyclical partition $\mathcal{V}_1$ of order at least $2 \eta(\mathcal{U})$. Otherwise we set $\mathcal{V}_1 := \mathcal{V}_0$. 
	The partition $\mathcal{V}_1$ refines $\mathcal{U}$ and $\mathcal{U}_n^{0}$, and is of order at least $\max(2\eta(\mathcal{U}),o(\mathcal{U})+1)$. Since these properties 
	are stable by refinement, the partitions constructed later in this proof ($\mathcal{V}_2$ and $\mathcal{V}$) will also have this property. 
	As a consequence of Theorem~\ref{theorem.krieger} the partition 
	$\mathcal{V}_1$ admits some $(s,(s+1)m+s,2s+1)$-marker denoted by $F$, where $m\ge 1$ and $s = 2 \eta(\mathcal{U})$.

\textbf{(2) Further refinements:} \textbf{\textit{(i)}} There exists an integer $n_1 \ge n+t$ such that for $l \le s$ each of the sets $f^{-l}(F)$ is
	the union of some elements in $\mathcal{U}_{n_1}^0$. 
	Since $F$ is clopen, the set $K=\bigcup_{l=0}^{s} f^{-l}(F)$ is also clopen, and as a consequence there exists some $\zeta>0$ such that if
	$d(x,f^{-l}(F))<\zeta$ for some $l\leq s$ then $x\in f^{-l}(F)$. We may also impose that $n_1$ is large enough so that for every $u\in \mathcal{U}_{n_1}^0$
we have $\diam f^l(u)<\zeta/(2s+2)$ for $l=0,\ldots, 2s+1$.
	We then consider 
	$\mathcal{V}_2=\mathcal{V}_1 \vee \mathcal{U}_{n_1}$.
\textbf{\textit{(ii)}} We perform one more modification on $\mathcal{V}_2$.
By definition of a marker, if we denote by $\mathcal{S}_1$ the supercyclical part of $\mathcal{V}_1$ then $F, f^{-1}(F), \ldots , f^{-2s-1}(F)$ cover the 
set $f^{-(s+1)m-s}(\E(\mathcal{S}_1))$. Since $\mathcal{E}(\mathcal{S}_2) \subset \mathcal{E}(\mathcal{S}_1)$ (by definition, the attracted part of $\mathcal{V}_2$ contains
the attracted part of $\mathcal{V}_1$), 
we have:
 \[f^{-(s+1)m-s}(\E(\mathcal{S}_2))\subset f^{-(s+1)m-s}(\E(\mathcal{S}_1)).\]
 
Furthermore, every point in $\E(\mathcal{S}_2)\setminus f^{-(s+1)m-s}(\E(\mathcal{S}_2))$
have to enter the attracted part of $\mathcal{V}_2$ after at most $(s+1)m+s$ iterations. Therefore, for sufficiently large integer $n_2$, the continuity of $f$
implies that any vertex $u\in \mathcal{U}^{0}_{n_2}$ either satisfies the inclusion $u\subset f^{-(s+1)m-s}(\E(\mathcal{S}_2))$ or $f^{(s+1)m+s}(u)\subset \E(\mathcal{V}_2\setminus \mathcal{S}_2)$. In particular, considering $\mathcal{V} = \mathcal{V}_2 \vee \mathcal{U}^{0}_{n_2}$, we have that $\E(\mathcal{S}) \subset f^{-(s+1)m-s}(\E(\mathcal{S}_2))$, and thus $\E(\mathcal{S}) \subset f^{-(s+1)m-s}(\E(\mathcal{S}_1))$, where $\mathcal{S}$ is the supercyclical part of $\mathcal{V}$.

We set $\tau':= o(\mathcal{V})$. Since $o(\mathcal{V}) \ge o(\mathcal{U})+1 \ge \tau+1$, we have $\tau' \ge \tau+1$. Furthermore, the supercyclical 
part of $(\mathcal{V},\tau')$ is equal to the one of $\mathcal{V}$, which is 
included in the supercyclical part of $\mathcal{U}$ and thus of $(\mathcal{U},\tau)$.
	
	\textbf{(3)} \textbf{Definition of $\chi'$ on the attracted part of $\mathcal{V}$:}
We claim 
	that there exists $\chi'$ such 
	that $(\mathcal{V},\tau',\chi')$ is well-marked 
	and well-marked relatively to $(\mathcal{U},\tau,\chi)$. We set $\chi'(u)=0$ for all $u$ 
	in the attracted part of $\mathcal{V}$. For every $u$ in the supercyclical part of $\mathcal{V}$, 
	if $\chi(\pi_{\mathcal{U}}^{\mathcal{V}}(u)) \in \{*,\downarrow\}$ or if $u$ is 
	not in any circuit contained in the supercyclical part of $\mathcal{V}$, set $\chi'(u)=\downarrow$.
	We are left to define $\chi'$ on preimages of potentials by 
	$\pi_{\mathcal{U}}^{\mathcal{V}}$ which are in a circuit of vertices in the supercyclical part of $\mathcal{V}$.

\textbf{(4) Circuits of the supercyclical part of $\mathcal{V}$ follow the sequence\break $f^{-s+1}(F), f^{-1}(F),..., F$ on some segment:} 
\textbf{(i) Claim:} By construction, for each of the circuits of vertices in the supercyclical part of $\mathcal{V}$, its vertices are included in 
$f^{-(s+1)m-s}(\E(\mathcal{S}_1))$ (since $\E(\mathcal{S}) \subset f^{-(s+1)m-s}(\E(\mathcal{S}_1))$). As a consequence (by the definition of marker)
	each of its vertices is included in some $f^{-k}(F)$ for $k \le 2s+1$.
	\textit{We claim}
	that each circuit has some sequence of consecutive vertices respectively included in 
	$f^{-s+1}(F), f^{-1}(F),..., F$ (it is possible to have other vertices outside of these sets).  
	\textbf{(ii) When the circuit intersects $F$:} \textit{Indeed}, for any vertex $u\subset F$ such that there 
	exists an edge $(v,u)$, there exists $x\in v$ such that $f(x)\in u\subset F$. This implies that $v\cap f^{-1}(F)\neq \emptyset$ and since $\diam v<\zeta$ we have $v\subset f^{-1}(F)$.
We apply a similar reasoning to $u \subset f^{-l+1}(F)$ for $l\leq (s-1)$: as a 
consequence the claim holds for any circuit which has a vertex included in $F$.
\textbf{(iii) General case:}
Let $c$ be a circuit whose vertices are included in the supercyclical part of $\mathcal{V}$. Since it is covered 
by the sets $f^{-l}(F)$ for $l\leq 2s+1$, there is some $l \le 2s+1$ 
and a vertex $u_0$ in the circuit $c$ such that $u_0\cap f^{-l}(F) \neq \emptyset$.
Let also  $u_1,\ldots, u_l$ be the next $l$ consecutive vertices in the circuit $c$. For all $i \le (l-1)$, there exists $x_i\in u_i$ such that $f(x_i)\in u_{i+1}$. Fix 
also a point $x \in u_0$ and $x_{l}\in u_l$.
Since $x,x_0\in u_0$ we have $d(f^l(x),f^l(x_0))<\zeta/(2s+2)$. Similarly for all $i \le (l-1)$, since $f(x_i),x_{i+1}\in u_{i+1}$ we have $d(f^{l-i}(x_i),f^{l-i-1}(x_{i+1}))<\zeta/(2s+2)$.
As a consequence we have:
\[
d(f^l(x),x_{l})\leq d(f^l(x),f^l(x_0))+\sum_{i=0}^{l-1}d(f^{l-i}(x_i),f^{l-i-1}(x_{i+1}))\leq \frac{\zeta(l+1)}{2s+2}\leq \zeta.\]
This implies that the distance between $x_l$ and $F$ is smaller than $\zeta$, and thus that $u_l\subset F$, which means that the claim holds for $c$.

\textbf{(5)} \textbf{Choosing markers and potentials in the preimages 
of potentials:} Let us consider a circuit $c$ in the supercyclical part of $\mathcal{V}$ and a sequence of consecutive vertices of $c$ contained respectively in $f^{-s+1}(F), f^{-1}(F),..., F$ - note that any two such sequences cannot have any common element. The first $\eta(\mathcal{U})$ of these vertices are mapped altogether via $\pi_{\mathcal{U}}^{\mathcal{V}}$ to a path which contains at least one circuit in 
	$G(\mathcal{U})$ whose vertices are in the supercyclical part of $\mathcal{U}$, by definition of $\eta(\mathcal{U})$. Since $(\mathcal{U},\tau,\chi)$ 
	is well-marked, each circuit in $G(\mathcal{U})$ whose vertices are in the 
	supercyclical part of $(\mathcal{U},\tau)$ contains at least one potential. Thus one of these vertices $\eta(\mathcal{U})$ is mapped to a potential. This is true also for the last
	$\eta(\mathcal{U})$ of these vertices 
	(let us recall that $s \ge 2\eta(\mathcal{U})$). 

 Let us list all the circuits that are in the supercyclical part of $\mathcal{V}$, say $c_1, ... , c_t$. Consider successively each of these circuits. Each time do the 
	following steps: 
	
	\begin{enumerate}
	\item If the first set of $\eta(\mathcal{U})$ 
	vertices defined above contains one for which $\chi'$ is already 
	defined to be $*$, then pass. 
	\item Otherwise pick one preimage $u$ of a potential in this set and define $\chi'(u)=*$.
	\item If the second set of vertices defined above contains one for which $\chi'$ is already 
	defined to be $\uparrow$, then pass. 
	\item Otherwise pick one preimage $u$ of a potential in this set and 
	define $\chi'(u)=\uparrow$.
	\end{enumerate}	
	For each vertex $u$ for which $\chi'$ is still undefined at the end of this process, 
	set $\chi'(u)=\downarrow$. By construction, since sets 	$f^{-s+1}(F), f^{-1}(F),..., F$ are pairwise disjoint, 
	it is straightforward that $(\mathcal{V},\chi')$ 
	is well-marked and well-marked relatively to $(\mathcal{U},\chi)$.
	\end{proof}	
	
	As a direct consequence of the construction in the proof of Lemma~\ref{lemma.well.marked} we also have the following.
	
	\begin{lemma}\label{lemma.existence.well.marked}
	There exists a well-marked partition $(\mathcal{U},\tau,\chi)$ for $(X,f)$ such that 
	$\mathcal{U} \prec \mathcal{U}_1^0$.
	\end{lemma}	
	
	\begin{remark}
	Before going further, let us explain why we needed to use Krieger's markers in order to 
	obtain Lemma~\ref{lemma.well.marked}. 
	
		\begin{figure}[h!]
		\begin{center}
		\begin{tikzpicture}[scale=0.2]
		
			\draw[fill=gray!98] (4,0) circle (5pt);
			\draw[fill=red] (0,4) circle (6pt);
			\draw[fill=red] (0,-4) circle (6pt);
			\draw[fill=gray!98] (-4,0) circle (5pt);

			\draw[-latex] (0.5,4) -- (4,0.5);
			\draw[-latex] (4,-0.5) -- (0.5,-4);
			\draw[-latex] (-0.5,-4) -- (-4,-0.5);
			\draw[-latex] (-4,0.5) -- (-0.5,4);
			\begin{scope}[yshift=-8.25cm]
			\draw[-latex] (0.5,4) -- (4,0.5);
			\draw[-latex] (4,-0.5) -- (0.5,-4);
			\draw[-latex] (-0.5,-4) -- (-4,-0.5);
			\draw[-latex] (-4,0.5) -- (-0.5,4);
			\draw[fill=red] (0,-4) circle (6pt);
			\draw[fill=gray!98] (-4,0) circle (5pt);
			\draw[fill=gray!98] (4,0) circle (5pt);
			\end{scope}
			
			\draw[fill=gray!98] (-8,-4) circle (5pt);
			\draw[fill=gray!98] (8,-4) circle (5pt);
			
			\draw[-latex] (-8,-3.5) to [out=90,in=180] (-0.5,4.25);
			\draw[-latex] (0.5,4.25) to [out=0,in=90] (8,-3.5);
			\draw[-latex] (8,-4.5) to [out=-90,in=0] (0.5,-12.5) ;
			\draw[-latex] (-0.5,-12.5)  to [out=180,in=-90] (-8,-4.5);
			
			\draw[dashed,-latex] (12,-4) -- (20,-4);
			\node at (16,-6) {$\pi$};
			\begin{scope}[xshift=25cm]
			\draw[fill=gray!98] (0,-2) circle (5pt);
			\draw[fill=red] (0,-6) circle (6pt);
			\node[scale=0.9] at (0,-8) {$\uparrow$};
			\node[scale=0.9] at (0,0) {$*$};
			\draw[-latex] (1,-2) to[out=-45,in=45] (1,-6);
			\draw[-latex] (-1,-6) to[out=135,in=-135] (-1,-2);
			\end{scope}
			
		\end{tikzpicture}
		\end{center}
		\caption{Two finite directed graphs $G=(V,E)$ (left) and $G'=(V',E')$ (right) 
		and $\pi: V\rightarrow V'$ morphism, sending vertices of $G$ to the one 
		of the same color in $G'$. The graph $G'$ can correspond to a well-marked partition 
		$(\mathcal{V},\tau,\chi)$. 
		However it is not possible to find a partition whose graph is $G$ and 
		which would be well-marked relatively to $(\mathcal{V},\tau,\chi)$. Indeed, whatever 
		the way we mark the red vertices, there will be at least one circuit left with 
		no marker or no potential.}\label{figure.example.marker.problem}
	\end{figure}
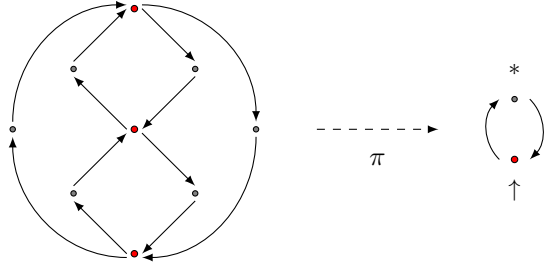

	Intuitively we could mark the partitions 
	using a simple argument which consists in noticing that it is possible to refine supercyclic 
	partitions so that circuits in the supercyclic part are uniformly arbitrarily large, 
	and thus are mapped by the graph morphism to a concatenation of at least two circuits of 
	the initial partition's graph. This way each circuit contains two preimages of potentials, 
	one of which we could define to be a marker and the other one a potential. The problem
	with this reasoning however is that circuits may have a lot of intersections, and 
	marking vertices in a circuit 
	could prevent the possibility of similar markings for other 
	circuits. This appears clear for the two graphs $G,G'$ on
	Figure~\ref{figure.example.marker.problem} and the morphism $\pi$ from $G$ to $G'$.
	In this figure $G$ could be replaced with some graphs 
	whose circuits could be taken uniformly arbitrarily 
	long (and having many more intersections). 
	\end{remark}

	\subsection{Rectification of a partition well marked relatively to 
	another one\label{section.rectification}}
	
	In this section we define the last operation on marked partitions in order 
	to ensure some other properties that will be needed for the graph coverings 
	representation used in the proof of Theorem~\ref{theorem.main}.

	\begin{definition}
	Let us consider $G$ a finite directed graph. We will call \textbf{acyclic cut} of $G$ 
	any set of its vertices such that by removing from $G$ all the edges that are pointing 
	at a vertex in this set, we obtain an acyclic graph.
	\end{definition}

	\begin{definition}
	Let us consider $G$ a finite directed graph. 
	A \textbf{divergent vertex} of 
	this graph is some vertex $u$ which has at least two outgoins edges.
	\end{definition}

	\begin{notation}
	Let us consider a finite directed graph $G$ and 
	$\chi \colon V \rightarrow \{\downarrow,\uparrow,0,*\}$, where $V$ is 
	the set of vertices of $G$. Let us denote by $\mathcal{R}(G,\chi)$ 
	the set whose elements are the following ones: every vertex $u \in V$ such that 
	$\chi(u)=0$ and such that there is no divergent vertex $v \in V$ with  
	$\chi(v) \neq 0$ with an edge from $v$ to $u$; the edges of $G$ pointing at 
	any of the vertices in $\mathcal{R}(G,\chi)$, and the edges from $u$ to $v$ 
	with $\chi(u)=\chi(v)=0$. 
	
	Let us denote by $\mathcal{I}(G,\chi)$ the graph 
	obtained from $G$ by removing the vertices and edges which are in $\mathcal{R}(G,\chi)$.
	We will also denote $\mathcal{A}(G,\chi)$ the graph 
	obtained from $\mathcal{I}(G,\chi)$ by adding, for each vertex $u$ of this graph 
	such that $\chi(u)=*$ (the function $\chi$ is defined on 
	vertices of $\mathcal{I}(G,\chi)$ by restriction) a copy $c_{u}$ of this vertex,
	and changing the edges pointing at $u$ so that they point at $c_u$ without changing 
	their origin. We also denote 
	$\mathcal{S}(G,\chi)$ the set of pairs $(u,c_{u})$ for $u$ vertex of $\mathcal{I}(G,\chi)$
	such that $\chi(u)=*$. The purpose of this set is to keep record of which vertex 
	is a copy of which vertex, for the reason that we will have to conflate 
	$u$ and $c_u$ at some point.
	This construction is illustrated on Figure~\ref{figure.acyclic.cutting}.
	\end{notation}

	\begin{remark}
	Considering a well-marked partition $(\mathcal{U},\tau,\chi)$, the set of its 
	markers forms an acyclic cut of $\mathcal{I}(G(\mathcal{U}),\chi)$. Thus the 
	graph $\mathcal{A}(G(\mathcal{U}),\chi)$ is acyclic. 
	\end{remark}
	
	\begin{figure}[h!]
	\begin{center}
			\begin{tikzpicture}[scale=0.2]
			\node at (-6,10) {$\boldsymbol{(G,\chi)}$};
			\draw[dashed,fill=gray!5] (-6,-5.5) rectangle (6,6.5);
			\draw[fill=gray!98] (0,-4) circle (5pt);
			\draw[fill=gray!98] (0,4) circle (5pt);
			\draw[fill=gray!98] (-4,0) circle (5pt);
			\draw[fill=gray!98] (4,0) circle (5pt);
			\draw[fill=gray!98] (2.8,2.8) circle (5pt);
			\draw[fill=gray!98] (2.8,-2.8) circle (5pt);
			\draw[fill=gray!98] (-2.8,2.8) circle (5pt);
			\draw[fill=gray!98] (-2.8,-2.8) circle (5pt);
			\draw[fill=gray!98] (-5.5,0) circle (5pt);
			\draw[fill=gray!98] (-7,0) circle (5pt);
			\draw[fill=gray!98] (4.3,2.8) circle (5pt);
			\draw[fill=gray!98] (2.8,4.3) circle (5pt);
			\draw[fill=gray!98] (2.8,5.8) circle (5pt);

			\draw[-latex] (0,4) -- (-2.8,2.8);
			\draw[-latex] (-2.8,2.8) -- (-4,0);
			\draw[-latex] (2.8,2.8) -- (0,4);
			\draw[-latex] (4,0) -- (2.8,2.8);
			\draw[latex-] (4,0) -- (2.8,-2.8);
			\draw[-latex] (0,-4) -- (2.8,-2.8);
			\draw[latex-] (0,-4) -- (-2.8,-2.8);
			\draw[-latex] (-4,0) -- (-2.8,-2.8);

			\draw[-latex] (-5.5,0) -- (-4,0);
			\draw[latex-] (-5.5,0) -- (-7,0);
			
			\draw[dashed,fill=gray!5] (-17,-19.5) -- (-17,-5) -- (-11,-5) -- (-11,-2) -- (-8.5,-2) -- (-8.5,-19.5) -- (-17,-19.5);

			\draw[latex-] (-10,-3) -- (-7,0);
			\draw[latex-] (-13,-6) -- (-10,-3);
			\draw[-latex] (-13,-6) -- (-13,-8);
			\draw[-latex] (-13,-8) -- (-13,-10);
			\draw[-latex] (-13,-10) -- (-13,-12);
			\draw[-latex] (-13,-12) -- (-10,-15);
			\draw[-latex] (-10,-15) -- (-13,-18);
			\draw[-latex] (-13,-18) -- (-16,-15);
			\draw[-latex] (-16,-15) -- (-13,-12);
			\draw[-latex] (-14.5,-4.5) -- (-13,-6);
			\draw[fill=gray!98] (-10,-3) circle (5pt);
			\draw[fill=gray!98] (-13,-6) circle (5pt);
			\draw[fill=gray!98] (-13,-8) circle (5pt);
			\draw[fill=gray!98] (-13,-10) circle (5pt);
			\draw[fill=gray!98] (-13,-12) circle (5pt);
			\draw[fill=gray!98] (-16,-15) circle (5pt);
			\draw[fill=gray!98] (-13,-18) circle (5pt);
			\draw[fill=gray!98] (-10,-15) circle (5pt);
			
			\draw[-latex] (-8.5,1.5) -- (-7,0);
			\draw[-latex] (-8.5,1.5) -- (-8.5,3.5);
			\draw[-latex] (-8.5,3.5) -- (-10.5,5.5);
			\draw[-latex] (-10.5,5.5) -- (-12.5,5.5);
			\node[scale=0.9] at (-12.5,6.5) {$*$};
			\draw[-latex] (-12.5,5.5) -- (-14.5,3.5);
			\draw[-latex] (-14.5,3.5) -- (-16.5,3.5);
			\draw[-latex] (-16.5,3.5) -- (-16.5,1.5);
			\draw[-latex] (-16.5,1.5) -- (-14.5,1.5);
			\draw[-latex] (-14.5,3.5) -- (-14.5,1.5);
			\draw[-latex] (-14.5,1.5) -- (-12.5,-0.5);
			\draw[-latex] (-12.5,-0.5) -- (-10.5,-0.5);
			\draw[-latex] (-10.5,-0.5) -- (-8.5,1.5);
			\draw[-latex] (-18.5,-0.5) -- (-16.5,1.5);
			\draw[-latex] (-18.5,-2.5) -- (-18.5,-0.5);
			\draw[-latex] (-16.5,-4.5) -- (-18.5,-2.5);
			\draw[-latex] (-14.5,-4.5) -- (-16.5,-4.5);
			\node[scale=0.9] at (-17.5,-5.5) {$*$};
			\draw[-latex] (-12.5,-2.5) -- (-14.5,-4.5);
			\draw[-latex] (-12.5,-0.5) -- (-12.5,-2.5);
			
			\draw[fill=gray!98] (-12.5,-2.5) circle (5pt);
			\draw[fill=gray!98] (-14.5,-4.5) circle (5pt);
			\draw[fill=gray!98] (-16.5,-4.5) circle (5pt);
			\draw[fill=gray!98] (-18.5,-2.5) circle (5pt);
			\draw[fill=gray!98] (-18.5,-0.5) circle (5pt);
			\draw[fill=gray!98] (-8.5,3.5) circle (5pt);
			\draw[fill=gray!98] (-10.5,5.5) circle (5pt);
			\draw[fill=gray!98] (-12.5,5.5) circle (5pt);
			\draw[fill=gray!98] (-14.5,3.5) circle (5pt);
			\draw[fill=gray!98] (-14.5,1.5) circle (5pt);
			\draw[fill=gray!98] (-12.5,-0.5) circle (5pt);
			\draw[fill=gray!98] (-10.5,-0.5) circle (5pt);
			\draw[fill=gray!98] (-8.5,1.5) circle (5pt);
			\draw[fill=gray!98] (-16.5,1.5) circle (5pt);
			\draw[fill=gray!98] (-16.5,3.5) circle (5pt);
			
			\draw[-latex] (2.8,4.3) -- (2.8,2.8);
			\draw[-latex] (4.3,2.8) -- (2.8,2.8);
			\draw[latex-] (2.8,4.3) -- (2.8,5.8);
			
			\begin{scope}[xshift=35cm]
			\node at (-12,10) {$\boldsymbol{\mathcal{I}(G,\chi)}$};
	
			\draw[fill=gray!98] (-5.5,0) circle (5pt);
			\draw[fill=gray!98] (-7,0) circle (5pt);

			\draw[latex-] (-5.5,0) -- (-7,0);

			\draw[latex-] (-10,-3) -- (-7,0);

			\draw[-latex] (-14.5,-4.5) -- (-13,-6);
			\draw[fill=gray!98] (-10,-3) circle (5pt);
			\draw[fill=gray!98] (-13,-6) circle (5pt);

			\draw[-latex] (-8.5,1.5) -- (-7,0);
			\draw[-latex] (-8.5,1.5) -- (-8.5,3.5);
			\draw[-latex] (-8.5,3.5) -- (-10.5,5.5);
			\draw[-latex] (-10.5,5.5) -- (-12.5,5.5);
			\node[scale=0.9] at (-12.5,6.5) {$*$};
			\draw[-latex] (-12.5,5.5) -- (-14.5,3.5);
			\draw[-latex] (-14.5,3.5) -- (-16.5,3.5);
			\draw[-latex] (-16.5,3.5) -- (-16.5,1.5);
			\draw[-latex] (-16.5,1.5) -- (-14.5,1.5);
			\draw[-latex] (-14.5,3.5) -- (-14.5,1.5);
			\draw[-latex] (-14.5,1.5) -- (-12.5,-0.5);
			\draw[-latex] (-12.5,-0.5) -- (-10.5,-0.5);
			\draw[-latex] (-10.5,-0.5) -- (-8.5,1.5);
			\draw[-latex] (-18.5,-0.5) -- (-16.5,1.5);
			\draw[-latex] (-18.5,-2.5) -- (-18.5,-0.5);
			\draw[-latex] (-16.5,-4.5) -- (-18.5,-2.5);
			\draw[-latex] (-14.5,-4.5) -- (-16.5,-4.5);
			\node[scale=0.9] at (-17.5,-5.5) {$*$};
			\draw[-latex] (-12.5,-2.5) -- (-14.5,-4.5);
			\draw[-latex] (-12.5,-0.5) -- (-12.5,-2.5);
			
			\draw[fill=gray!98] (-12.5,-2.5) circle (5pt);
			\draw[fill=gray!98] (-14.5,-4.5) circle (5pt);
			\draw[fill=gray!98] (-16.5,-4.5) circle (5pt);
			\draw[fill=gray!98] (-18.5,-2.5) circle (5pt);
			\draw[fill=gray!98] (-18.5,-0.5) circle (5pt);
			\draw[fill=gray!98] (-8.5,3.5) circle (5pt);
			\draw[fill=gray!98] (-10.5,5.5) circle (5pt);
			\draw[fill=gray!98] (-12.5,5.5) circle (5pt);
			\draw[fill=gray!98] (-14.5,3.5) circle (5pt);
			\draw[fill=gray!98] (-14.5,1.5) circle (5pt);
			\draw[fill=gray!98] (-12.5,-0.5) circle (5pt);
			\draw[fill=gray!98] (-10.5,-0.5) circle (5pt);
			\draw[fill=gray!98] (-8.5,1.5) circle (5pt);
			\draw[fill=gray!98] (-16.5,1.5) circle (5pt);
			\draw[fill=gray!98] (-16.5,3.5) circle (5pt);
			\end{scope}
			
			\begin{scope}[xshift=35cm,yshift=-25cm]
			\node at (-12,10) {$\boldsymbol{\mathcal{A}(G,\chi)}$};
	
			\draw[fill=gray!98] (-5.5,0) circle (5pt);
			\draw[fill=gray!98] (-7,0) circle (5pt);
			\draw[latex-] (-5.5,0) -- (-7,0);
			\draw[latex-] (-10,-3) -- (-7,0);

			\draw[-latex] (-14.5,-4.5) -- (-13,-6);
			\draw[fill=gray!98] (-10,-3) circle (5pt);
			\draw[fill=gray!98] (-13,-6) circle (5pt);
			
			\draw[-latex] (-8.5,1.5) -- (-7,0);
			\draw[-latex] (-8.5,1.5) -- (-8.5,3.5);
			\draw[-latex] (-8.5,3.5) -- (-10.5,5.5);
			\draw[-latex] (-10.5,5.5) -- (-12.5,7.5);
			\draw[fill=gray!98] (-12.5,7.5) circle (5pt);
			\draw[-latex] (-12.5,5.5) -- (-14.5,3.5);
			\draw[-latex] (-14.5,3.5) -- (-16.5,3.5);
			\draw[-latex] (-16.5,3.5) -- (-16.5,1.5);
			\draw[-latex] (-16.5,1.5) -- (-14.5,1.5);
			\draw[-latex] (-14.5,3.5) -- (-14.5,1.5);
			\draw[-latex] (-14.5,1.5) -- (-12.5,-0.5);
			\draw[-latex] (-12.5,-0.5) -- (-10.5,-0.5);
			\draw[-latex] (-10.5,-0.5) -- (-8.5,1.5);
			\draw[-latex] (-18.5,-0.5) -- (-16.5,1.5);
			\draw[-latex] (-18.5,-2.5) -- (-18.5,-0.5);
			\draw[-latex] (-16.5,-4.5) -- (-18.5,-2.5);
			\draw[-latex] (-14.5,-4.5) -- (-16.5,-6.5);
			\draw[fill=gray!98] (-16.5,-6.5) circle (5pt);
			\draw[-latex] (-12.5,-2.5) -- (-14.5,-4.5);
			\draw[-latex] (-12.5,-0.5) -- (-12.5,-2.5);
			\draw[-latex] (-10,-3) -- (-13,-6);
			
			\draw[fill=gray!98] (-12.5,-2.5) circle (5pt);
			\draw[fill=gray!98] (-14.5,-4.5) circle (5pt);
			\draw[fill=gray!98] (-16.5,-4.5) circle (5pt);
			\draw[fill=gray!98] (-18.5,-2.5) circle (5pt);
			\draw[fill=gray!98] (-18.5,-0.5) circle (5pt);
			\draw[fill=gray!98] (-8.5,3.5) circle (5pt);
			\draw[fill=gray!98] (-10.5,5.5) circle (5pt);
			\draw[fill=gray!98] (-12.5,5.5) circle (5pt);
			\draw[fill=gray!98] (-14.5,3.5) circle (5pt);
			\draw[fill=gray!98] (-14.5,1.5) circle (5pt);
			\draw[fill=gray!98] (-12.5,-0.5) circle (5pt);
			\draw[fill=gray!98] (-10.5,-0.5) circle (5pt);
			\draw[fill=gray!98] (-8.5,1.5) circle (5pt);
			\draw[fill=gray!98] (-16.5,1.5) circle (5pt);
			\draw[fill=gray!98] (-16.5,3.5) circle (5pt);
			\end{scope}
			\end{tikzpicture}
		\end{center}
		\caption{Illustration on an example of the definition of the graphs 
		$\mathcal{I}(G,\chi)$ 
		and $\mathcal{A}(G,\chi)$ for the graph $G=G(\mathcal{U})$ and 
		$\chi: \mathcal{U} \rightarrow \{0,\downarrow,\uparrow,*\}$
		where $(\mathcal{U},\chi)$ is a well-marked partition; the function $\chi$ is 
		partially represented (for simplicity): only markers and 
		vertices with $\chi(u)=0$ (the ones in dashed regions) are represented.}\label{figure.acyclic.cutting}
	\end{figure}
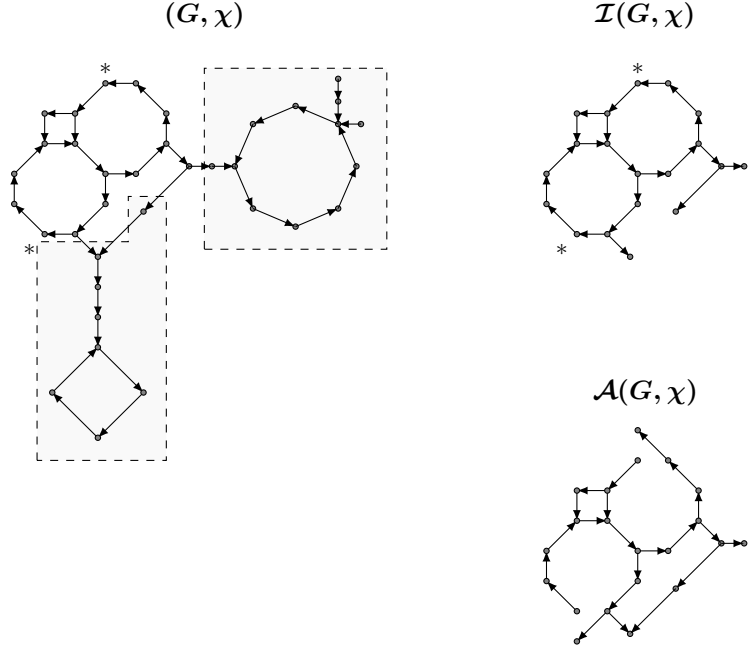

	For a supercyclical partition $\mathcal{U}$ and any integer $n \ge 1$, 
	every divergent element of $G(\kappa_n(\mathcal{U}))$ belongs to the supercyclical 
	part of $\kappa_n(\mathcal{U})$. The specificity of purely attracting 
	Cantor systems, besides the fact 
	that supercyclical partitions segregate circuits (no circuit crosses both the attracted 
	part and the supercyclical one) lies in this property, which will allow us, by refining 
	partitions, to "move" all divergent elements so that 
	they coincide with markers.

	\begin{definition}
	For an acyclic directed graph, we call \textbf{initial vertex} any vertex of this 
	graph which has no edge pointing at it. 
	\end{definition}
	
	\begin{notation}
	Let us consider $\mathcal{A}$ an acyclic directed graph. We will denote $V(\mathcal{A})$ its vertex set. Consider 
	two vertices $v$ and $v'$ of $\mathcal{A}$ such that there is a path from $v$ to $v'$ 
	in $\mathcal{A}$. We call distance - which is not a distance 
	in the usual sense - between $v$ and $v'$ the minimal length of 
	a path from $v$ to $v'$. We will denote $\delta(\mathcal{A})$ the maximal distance between 
	an initial point and a divergent point of $\mathcal{A}$, and we use 
	the convention $\delta(\mathcal{A})=0$ when there is no divergent point in $\mathcal{A}$. We will also denote 
	$\mu(\mathcal{A})$ the number of divergent vertices which realize this maximum.
	\end{notation}
	
	Let us observe that for an acyclic directed graph $\mathcal{A}$, this graph possesses 
	a divergent point which is not initial if and only if $\delta(\mathcal{A})>0$. 

	\begin{definition}\label{definition.divergent.displace}
	Let us consider a directed acyclic graph $\mathcal{A}$ such that $\delta(\mathcal{A})>0$
	and $v$ a divergent 
	vertex realizing the maximum in the definition of $\delta(\mathcal{A})$. 
	We consider the acyclic 
	graph $\accentset{v}{\mathcal{A}}$ obtained 
	from $\mathcal{A}$ after the following modifications: (i) First replace the vertex $v$ 
	with the set of vertices $v \cap f^{-1}(w)$, where there is an edge pointing 
	from $v$ to $w$; (ii) Replace all the edges pointing to $v$ by a set of edges pointing 
	at the constructed vertices $v \cap f^{-1}(w)$
	according to the behaviour of the system $(X,f)$.
	\end{definition}
	
	Let us denote $\le_{\text{lex}}$ the lexicographic order on $\mathbb{R}^2$. The following 
	is straightforward:
	
	\begin{lemma}\label{lemma.lexical}
	With notations from Definition~\ref{definition.divergent.displace}, we have 
	that: 
	\[(\delta(\accentset{v}{\mathcal{A}}),\mu(\accentset{v}{\mathcal{A}})) {<}_{\text{lex}} (\delta(\mathcal{A}),\mu(\mathcal{A})).\]
	\end{lemma}
	
	\begin{notation}
	For an acyclic graph such that $\delta(\mathcal{A})>0$ and $\chi : V(\mathcal{A}) \rightarrow 
	\{0,*,\uparrow,\downarrow\}$, we denote $\accentset{v}{\chi} : V(\accentset{v}{\mathcal{A}})
	\rightarrow \{0,*,\uparrow,\downarrow\}$ such that for every vertex created 
	in the definition of $\accentset{v}{\mathcal{A}}$, the value of $\accentset{v}{\chi}$ 
	is equal to the value of $\chi$ on the vertex that they replace. On the other 
	vertices the function $\accentset{v}{\chi}$ coincides with $\chi$.
	\end{notation}
	
	\begin{notation}
	For simplicity, for a well-marked partition $(\mathcal{V},\tau,\chi)$, we will 
	denote $\delta(\mathcal{V},\chi)$ and $\mu(\mathcal{V},\chi)$ the respective numbers 
	$\delta(\mathcal{A}(G(\mathcal{V}),\chi))$ and
	$\mu(\mathcal{A}(G(\mathcal{V}),\chi))$.
	\end{notation}
	
	\subsubsection{Decreasing the value of $(\delta,\mu)$ for a well-marked partition}

	Let us consider a well-marked 
	partition $(\mathcal{U},\tau,\chi)$ and 
	$\mathcal{A}$ the acyclic graph $\mathcal{A}(G(\mathcal{U}),\chi)$, 
	and assume that $\delta(\mathcal{A})>0$. Let us denote $\chi_{\mathcal{A}}$ 
	the function $V(\mathcal{A}) \rightarrow \{\uparrow,\downarrow,0,*\}$ such that: 
	\begin{enumerate}[(i)]
	\item $\chi_{\mathcal{A}}$ is identical to $\chi$ on the vertices of 
	$\mathcal{I}(G(\mathcal{U}),\chi)$ on which $\chi$ is defined; 
	\item $\chi_A (c_u)=*$ whenever 
	$(u,c_u) \in \mathcal{S}(G(\mathcal{U}),\chi)$. 	
	\end{enumerate}
In other words we add mark $*$ to all newly created vertexes $c_u$.

	Consider $v$ a vertex which realizes the maximum 
	in the definition of $\delta(\mathcal{A})$. 
	Let us modify the graph $G(\mathcal{U})$ using $\accentset{v}{\mathcal{A}}$. Strictly speaking, we construct a new graph $G$ as follows:
	\begin{enumerate}
	\item for all 
	$(u,c_u) \in \mathcal{S}(G(\mathcal{U}),\chi)$, remove from $\accentset{v}{\mathcal{A}}$ the vertex $c_u$ from the graph and 
	change the edges pointing at $c_u$ so that they point at $u$, without changing their 
	origin (this is possible because divergent vertices are left unchanged in the transformation 
	of $\mathcal{A}$ into $\accentset{v}{\mathcal{A}}$). 
	\item add the vertices and edges of the set $\mathcal{R}(G(\mathcal{U}),\chi)$ (this operation 
	is possible since the origin vertices of these edges are unchanged by the transformation of
	$\mathcal{A}$ into $\accentset{v}{\mathcal{A}}$). 
	\end{enumerate}
In other words we glue back vertexes marked $*$ who were split before according to the relation $\mathcal{S}(G(\mathcal{U}),\chi)$.
We also have to define a modification of function $\chi$, since the set of vertices changed. Let $V$ denote the set 
of vertices of $G$.
Define $\chi' \colon V \rightarrow \{\uparrow,\downarrow,*,0\}$,  such that $\chi'$ coincide with $\accentset{v}{\chi_{\mathcal{A}}}$ on 
the graph obtained out of $\accentset{v}{\mathcal{A}}$ is step (1), and 
	with value $0$ on the vertices added in step (2).
	
	\begin{lemma}\label{lemma.descending.distance}
	There exists a partition $\mathcal{V}$ which refines $\mathcal{U}$ such 
	that $G(\mathcal{V})=G$. Moreover $(\mathcal{V},\tau,\chi')$ is well-marked and 
	for every vertex $u$ of $G(\mathcal{V})$, we have
	$\chi'(u) = \chi(\pi_{\mathcal{U}}^{\mathcal{V}}(u))$. Furthermore: 
	\[(\delta(\mathcal{V},\chi'),\mu(\mathcal{V},\chi')) <_{\text{lex}} (\delta(\mathcal{U},\chi),\mu(\mathcal{U},\chi)).\]
	\end{lemma}
	
	\begin{proof}
	Indeed the vertices created in the transformation of $\mathcal{A}$ into 
	$\accentset{v}{\mathcal{A}}$
	form a partition of the vertex 
	they replace. The collection of vertices of the graph $G$ 
	thus forms a finite clopen partition $\mathcal{V}$ of $X$ and 
	this partition refines $\mathcal{U}$. Hence $G=G(\mathcal{V})$. By 
	construction of $\accentset{v}{\chi_{\mathcal{A}}}$, we have 
	$\chi'(u) = \chi(\pi_{\mathcal{U}}^{\mathcal{V}}(u))$ for all $u$ vertex of $G$. 
	We have that $(\mathcal{V},\tau,\chi')$ is well-marked. This comes 
	from the two facts that the image 
	of every circuit of 
	$G(\mathcal{V})$ by $\pi_{\mathcal{U}}^{\mathcal{V}}$ contains at least one 
	circuit of $G(\mathcal{U})$, 
	and that for all $u$, $\chi'(u) = \chi(\pi_{\mathcal{U}}^{\mathcal{V}}(u))$. 
	The consequence of these facts is that every circuit in the supercyclical 
	part of $(\mathcal{V},\tau)$ contains at least one marker and one potential.
	The last part of the lemma is a direct consequence of Lemma~\ref{lemma.lexical}.
	\end{proof}
	
	\begin{theorem}\label{theorem.rectification}
	Let us consider a well marked partition $(\mathcal{U},\tau,\chi)$. 
	There exists another well marked partition $(\mathcal{V},\tau,\chi')$ such that $\mathcal{V}$ 
	refines $\mathcal{U}$ and such that divergent points of $\mathcal{V}$ are markers 
	for $(\mathcal{V},\tau,\chi')$ and all these markers are mapped to markers of 
	$(\mathcal{U},\tau,\chi)$ by $\pi_{\mathcal{U}}^{\mathcal{V}}$.
	\end{theorem}
	
	\begin{proof}
	
	Let us construct a sequence of well-marked partitions $(\mathcal{U}_m,\tau,\chi_m)_{m \ge 0}$ 
	recursively as follows: 
	
	\begin{enumerate}
	\item  the first element of this sequence is given by
	$\mathcal{U}_0 = \kappa_1(\mathcal{U})$, which refines $\mathcal{U}$ and $\chi_0 : 
	\mathcal{U}_0 \rightarrow \{\uparrow,0,*,\downarrow\}$ 
	which coincides with $\chi$ on the supercyclical part of $\mathcal{U}_0$ and takes 
	constant value $0$ on the attracted part of $\mathcal{U}_0$. We have directly that 
	$(\mathcal{U}_0,\tau,\chi_0)$ is well marked. 
	\item for all $m \ge 0$, if 
	$\delta(\mathcal{U}_m,\chi_m)>0$, $(\mathcal{U}_{m+1},\chi_{m+1})$ is 
	the well-marked partition obtained from $(\mathcal{U}_m,\tau,\chi_m)$ 
	by Lemma~\ref{lemma.descending.distance}. In this case we have that 
	\[(\delta(\mathcal{U}_{m+1},\chi_{m+1}),\mu(\mathcal{U}_{m+1},\chi_{m+1}))<_{\text{lex}} 
	(\delta(\mathcal{U}_{m},\chi_m),\mu(\mathcal{U}_{m},\chi_m)).\] Otherwise set
	$(\mathcal{U}_{m+1},\tau,\chi_{m+1})$ equal to $(\mathcal{U}_{m},\tau,\chi_{m})$. By this construction for all $m$, $\mathcal{U}_{m+1}$ refines $\mathcal{U}_{m}$ and 
	$(\mathcal{U}_{m},\tau,\chi_{m})$ is well marked.
	\end{enumerate}

	By infinite descent argument, 
	there exists some $m_0 \ge 0$ (minimal) such that 
	$\delta(\mathcal{U}_{m_0},\chi_{m_0})=0$. This means that every divergent 
	vertex in the graph of $\mathcal{U}_{m_0}$ is initial in this graph 
	or is a marker. By refining one more time we can remove the initial divergent
	vertices by splitting them in as many vertices as outgoing arrows. We denote $\mathcal{V}$ the obtained partition, and $\chi' : V(\mathcal{V}) \rightarrow \{\uparrow,\downarrow,0,*\}$ obtained by attributing, 
	for each splitted vertex, its value for $\chi_{m_0}$ to the vertices 
	introduced after splitting.
	
	The marked partition $(\mathcal{V},\tau,\chi')$ is well-marked (indeed $(\mathcal{U}_{m_0},\tau,\chi_{m_0})$ is well-marked and the construction of $(\mathcal{V},\tau,\chi')$ from 
	it does not modify the circuits)
	and $\mathcal{V} \prec \mathcal{U}$. By construction 
	every divergent vertex of the graph $G(\mathcal{V})$ is a marker 
	of $(\mathcal{V},\tau,\chi')$. Moreover Lemma~\ref{lemma.descending.distance} implies that 
	for each vertex $u$ of $G(\mathcal{V})$,
	
	\[\chi'(u)  = \chi \left( \pi_{\mathcal{U}}^{\mathcal{U}_{0}}  \circ \pi_{\mathcal{U}_{0}}^{\mathcal{U}_{1}} \circ \hdots  \circ \pi_{\mathcal{U}_{m_0-1}}^{\mathcal{U}_{m_0}} \circ \pi_{\mathcal{U}_{m_0}}^{\mathcal{V}}(u)\right) = \chi( \pi_{\mathcal{U}}^{\mathcal{V}}(u)).\] 
This implies that markers of $(\mathcal{V},\tau,\chi')$ are mapped to markers 
	of $(\mathcal{U},\tau,\chi)$ by $\pi_{\mathcal{U}}^{\mathcal{V}}$. This concludes 
	the proof. 
	\end{proof}

The partition obtained by the above theorem has the following property:
\begin{enumerate}
	\item[(S4)] 
	the partition is well marked and all divergent vertexes are markers.
\end{enumerate}

	\begin{theorem}\label{theorem.partition.sequence}
	The system $(X,f)$ admits a sequence of finite clopen partitions 
	$\mathbb{U}^{*}=(\mathcal{U}_n)_{n \ge 1}$ such that 
	$\lim_{n\to\infty} \mesh(\mathcal{U}_n) = 0$ and for all $n \ge 1$ and for 
	all $u \in \mathcal{U}_n$, $u$ is equal to the set $\mathcal{E}(\{v \in \mathcal{U}_{n+1} : v \subset u\})$, 
	and a sequence of functions 
	$\chi_n : \mathcal{U}_n \rightarrow \{\downarrow,0,*,\uparrow\}$ and 
	a non-decreasing sequence of integers $\tau_n \ge n$ such that 
	for all $n \ge 0$, $(\mathcal{U}_n,\tau_n,\chi_n)$ is a well-marked partition, 
	and $(\mathcal{U}_{n+1},\tau_{n+1},\chi_{n+1})$ is well marked relatively to 
	$(\mathcal{U}_{n},\tau_n,\chi_{n})$. Moreover every divergent point of $\mathcal{U}_{n}$
	is a marker of $(\mathcal{U}_{n},\tau_n,\chi_n)$.
	\end{theorem}
	
	\begin{proof}
	This result can be derived from all the previous steps leading to conditions (S1)-(S4)
	together with keeping that each of the partition is well marked with respect to the previous one.
	Formally, it is obtained as follows.
	
	First we use Lemma~\ref{lemma.existence.well.marked} to get a well marked partition 
	for $(X,f)$, and define $(\mathcal{U}_1,\tau_1,\chi_1)$ to be this partition. Let us assume 
	that we have constructed $(\mathcal{U}_k,\tau_k,\chi_k)$ for all $k \le n$ for some $n\ge 1$. 
	We then apply Theorem~\ref{theorem.rectification} on the well-marked partition obtained 
	from Lemma~\ref{lemma.well.marked} with integer $n+1$, 
	and set $(\mathcal{U}_{n+1},\tau_{n+1},\chi_{n+1})$ 
	to be the obtained marked partition. By construction this partition is well-marked and
	well-marked relatively to $(\mathcal{U}_n,\tau_n,\chi_n)$, and $\tau_{n+1} \ge \tau_n +1$, 
	which implies that $\tau_n \ge n$ for all $n$. Also every divergent 
	point of $\mathcal{U}_{n+1}$ is a marker of $(\mathcal{U}_{n+1},\tau_{n+1},\chi_{n+1})$. Moreover 
	we have that for all $n \ge 1$, $\mathcal{U}_n \prec \mathcal{U}_{n}^0$ and 
	$\mathcal{U}_{n+1} \prec \mathcal{U}_{n}$, which imples that $\mathbb{U}^{*}$ 
	satifies the condition \eqref{eq:star}.
	\end{proof}
	
	A direct consequence of properties of $\mathbb{U}^{*}$ constructed in the proof 
	of Theorem~\ref{theorem.partition.sequence} is that for every 
	$\textbf{v} \in V_{\mathbb{U}^{*},f}$, there exists at most one integer $n \ge 1$ 
	such that $\chi_n(\textbf{v}_n)=*$. Moreover if $\textbf{v}_n$ is divergent in 
	$G(\mathcal{U}_n)$, then $\chi_n(\textbf{v}_n)=*$. These two properties 
will be crucial in the proof of main theorem.

\section{On the embedding problem : proof of Theorem~\ref{theorem.main}\label{section.proof}}

In this section we will provide a proof of Theorem~\ref{theorem.main}, that we recall here
(the notion of attracting finite orbit is provided 
in Definition~\ref{definition.attracting.orbit}):

\begin{T1}
Any Cantor dynamical system $(X,f)$ can be embedded in the interval
$[0,1]$ with vanishing derivative if and only if it is purely attracting.
\end{T1}

The implication $(\Leftarrow)$ builds on all previous construction and its proof is presented in Section~\ref{section.sufficient}. 
The converse is an immediate consequence
of the following simple observations.
	
\begin{lemma}\label{lemma.derivative}
		Let us assume that $X$ is a Cantor set and $f$ is differentiable on $X$. Any finite orbit $p$ of the system $(X,f)$ such that for all $x \in p$, $f'(x)=0$ is attracting.
	\end{lemma}

\begin{cor}\label{cor:attracting}
If a Cantor dynamical system $(X,f)$ can be embedded in the real line with vanishing derivative then it is purely attracting.
\end{cor}
	
\begin{proof}
Assume that $(X,f)$ can be embedded in the real line with vanishing 
derivatie. This means that there exist $Z\subset \mathbb{R}$, $g\colon Z\to Z$ with 
$g' \equiv 0$ and $\psi \colon X \rightarrow Z$ an embedding which conjugates $(Z,g)$ and $(X,f)$. Since $g' \equiv 0$, Lemma~\ref{lemma.derivative} implies that all periodic orbits of $g$ are attracting. Attracting orbits are preserved under conjugacy, thus $(X,f)$ is purely attracting, which completes the proof.
\end{proof}

\subsection{Sufficiency of the condition\label{section.sufficient}}

The embedding required in Theorem~\ref{theorem.main} will be a consequence of the following theorem obtained first in \cite{Jarnik} (see \cite{CiCi} for a contemporary proof of this result). It will allow us to reduce
the problem to construction of a metric on Cantor set.
		
		\begin{theorem}[Jarn\'\i{}k]\label{theorem.jarnik}
			Let $X \subset \mathbb{R}$ be a perfect set and $f \colon X \rightarrow \mathbb{R}$ differentiable.
			Then there exists a differentiable extension $\tau \colon \mathbb{R} \rightarrow \mathbb{R}$ of $f$.
		\end{theorem}

Set a purely attracting Cantor dynamical system $(X,f)$. 
Let us consider $\mathbb{U}^{*}=(\mathcal{U}_n)_{n\ge 1}$ and $(\chi_n)_{n \ge 1}$
and $(\tau_n)_{n \ge 1}$ 
obtained with Theorem~\ref{theorem.partition.sequence} for this system. Let us also 
consider $(\textbf{G},\pi)$ the graph covering representation which corresponds 
to $\mathbb{U}^{*}$ for $(X,f)$. 

For each $n \ge 1$, we define an acyclic graph $A_n$ obtained from $G_n=G(\mathcal{U}_n)$ by: 
\begin{enumerate}[(i)]
\item removing edges that are pointing at a marker (a vertex $u$ of $G_n$ such that 
$\chi_n(u)=*$); 
\item removing the vertices in circuits included in the attracted part of $(\mathcal{U}_n,\tau_n)$
and edges pointing at any of these vertices.
\end{enumerate}

We also fix a total order on every set $\pi_n^{-1}(v)$ for $v \in V_{n}$ 
and $n \ge 1$, where $V_n$ is the vertex set of $G_n$ and $\pi_n\colon V_{n+1}\to V_n$ is associated bonding map. 

\subsubsection{Contraction rates} 

Fix $n \ge 1$ and 
a finite orbit $p$ in the attracted part of $\mathcal{U}_n$. By Lemma~\ref{lemma.graph.supercyclical}, for all $u \in \leftidx{^p}{\mathcal{U}_n}$, there is a unique path in $G(\mathcal{U}_n)$ starting from $u$ and ending 
on an element of the circuit corresponding to $p$. We will denote by $\delta_n (u)$ the number of edges in this path.

The attracted part of $(\mathcal{U}_{n},\tau_n)$ contains the attracted part of 
$(\mathcal{U}_{n-1},\tau_{n-1})$. Indeed, let us remind that the attracted part 
of $(\mathcal{U}_{n},\tau_n)$ is the union of the sets $\leftidx{_*^p}{\mathcal{U}_{n}}$ for $|p| \le \tau_n$. Since $\tau_n \ge \tau_{n-1} +1$, it contains 
in particular the sets $\leftidx{_*^p}{\mathcal{U}_{n}}$ for $|p| \le \tau_{n-1}$.
Since $\mathcal{U}_{n}$ refines $\mathcal{U}_{n-1}$, for $|p| \le \tau_{n-1}$, 
$\leftidx{_*^p}{\mathcal{U}_{n-1}} \subset \leftidx{_*^p}{\mathcal{U}_{n}}$. 
This implies that the attracted part of $\mathcal{U}_{n-1}$ is contained in the 
attracted part of $\mathcal{U}_{n}$.

As a consequence every preimage by $\pi_{n-1}$ of a vertex in the circuit 
corresponding to a finite orbit $p$ in the attracted part of $(\mathcal{U}_{n-1},\tau_{n-1})$ 
is contained in the attracted part of 
$(\mathcal{U}_{n},\tau_n)$. Thus $\delta_{n}$ is defined for these vertices. 
We will denote by $\omega_{n}(p)$ the maximum of $\delta_{n}$ on $\leftidx{^p}{\mathcal{U}_n}$.

For all $n \ge 1$ we define the shrinking rate $\lambda_n \in (0,1)$ by: 
\[\lambda_n = \frac{1}{2^n} \cdot \frac{1}{|V_{n+1}|+1}\] 
Additionally, we define an auxiliary length $\epsilon_n$ by putting $\epsilon_1 = 1$ and then inductively for all $n \ge 1$: 
\[\epsilon_{n+1} = \frac{1}{4 \cdot 2^{n |E_n|}} \cdot \frac{\lambda_n^{|E_{n}|}}{|V_{n+1}|+1} \cdot \epsilon_n\]

\subsubsection{From the graph representation to interval partitions\label{section.setting}}

In the following, we will denote by $\ell(I)$ the length of a compact interval $I$ and by $\rho(I)$ its middle point. Let us notice that the data of $\ell(I)$ together with $\rho(I)$ determines 
completely the interval $I$. 

We construct a function $\iota : V_{\textbf{G},\pi} \rightarrow \mathcal{I}^{\mathbb{N}_{+}}$, 
where $\mathcal{I}$ is the set of compact intervals of $\mathbb{R}$, by defining functions $\iota_n : V_n \rightarrow \mathcal{I}$ such that $\iota(\textbf{v})_n = \iota_n(\textbf{v}_n)$
for all $n \ge 1$ and such that for each $v \in V_{n+1}$ we have inclusion $\iota_{n+1}(v) \subset \iota_n(\pi_n(v))$. 
We define this sequence of functions recursively. 
In order to define $\iota_n$ for all $n \ge 1$, we 
define the length $\ell(\iota_{n}(v))$ for all $v$ and then the middle point $\rho(\iota_{n}(v))$.

If $v$ is a vertex of $G_n$ which is an initial vertex in the acyclic graph $A_n$, we set 
$\ell(\iota_n(v))=\epsilon_n$, and determine $\ell(\iota_n(v))$ on 
all the other vertices $v$ of $G_n$ that are in $A_n$ 
by imposing that for two vertices $u,v \in V_n$ such that $(u,v)$ is an 
edge in $A_n$ we denote by 
$S_v=\{w : (w,v) \text{ is an edge of }A_n\}$ and
$$
\ell(\iota_n(v))=\min_{w\in S_v} \lambda_n \cdot \ell(\iota_n(w)).
$$ 
This condition can be easily ensured by recursion.
Roughly speaking,
the length of the interval $\iota_n(v)$ ``shrinks'' as one goes along any path in $A_n$ with shrinking rate at least $\lambda_n$.
On every vertex $v$ which is in a circuit $c$ corresponding to a finite orbit 
$p$ included in the attracted part of $\mathcal{U}_n$, 
we set $\ell(\iota_n(v))$ to be  $\lambda_n \cdot m_n (p)$, where 
$m_n(p)$ is the minimum of the numbers $\ell(\iota_n(u))$ where the vertex $u$ does not belong to the circuit $c$ but
there is $w$ in $c$ and edge $(u,w)$ belongs to $G_n$. 

For $n=1$, we can choose freely the middle points $\rho(\iota_1(v))$ for $v \in V_1$. We only have to ensure that 
any two intervals $\iota_1(v)$ have empty intersection.
When $n \ge 2$, in order to define the middle point of the interval 
$\iota_n(v)$ for each vertex $v$, 
we will determine its relative position in the set $\iota_{n-1}(\pi_{n-1}(v))$.
It will to some extent rely on the ordering in the set
$\pi_{n-1}^{-1}(w)$ 
where $w=\pi_{n-1}(v)\in V_{n-1}$. 

Let us fix any $w \in V_{n-1}$. We distinguish two cases: 

\begin{enumerate}
\item \textbf{When $w$ is outside of any circuit included in the attracted part of $\mathcal{U}_{n-1}$:}

For $k$ such that $v$ is the $k$th of elements of $\pi_{n-1}^{-1}(w)$, see Figure~\ref{figure.cutting.process}:
\[\rho(\iota(v)) = \left(\rho(\iota(w)) - \frac{\ell(\iota(w))}{2}\right) + k \cdot \frac{\ell(\iota(w))}{|\pi_{n-1}^{-1}(w)|+1}.\]

\begin{figure}[h!]
\begin{center}
\begin{tikzpicture}[scale=0.3]

\draw[latex-latex,color=gray!70] (0,3) -- (15,3);
\node[scale=0.9] at (12.5,4) {$\ell(\iota_{n-1}(w))$};
\draw (0,0) -- (15,0);
\draw[line width=0.4mm] (7.5,-0.25) -- (7.5,0.25);
\node[scale=0.9] at (7.5,1.5) {$\rho(\iota_{n-1}(w))$};
\draw[line width=0.3mm] (0,-0.25) -- (0,0.25);
\draw[line width=0.3mm] (15,-0.25) -- (15,0.25);

\begin{scope}[xshift=-12cm]
\draw[fill=gray!90] (0,-1) circle (5pt);
\node[scale=0.9] at (-8,-4) {$\pi_{n-1}^{-1}(w)$};
\node[scale=0.9] at (-7,-1) {$w$};
\draw[-latex] (-1.25,-4) -- (-0.075,-1.25);
\draw[-latex] (-3.75,-4) -- (-0.25,-1.25);
\draw[-latex] (1.25,-4) -- (0.075,-1.25);
\draw[-latex] (3.75,-4) -- (0.25,-1.25);
\draw[fill=gray!90] (-1.25,-4) circle (5pt);
\draw[fill=gray!90] (-3.75,-4) circle (5pt);
\draw[fill=gray!90] (1.25,-4) circle (5pt);
\draw[fill=gray!90] (3.75,-4) circle (5pt);
\end{scope}

\begin{scope}[yshift=-5cm]
\draw (2.5,0) -- (3.5,0);
\draw[line width=0.3mm] (2.5,-0.25) -- (2.5,0.25);
\draw[line width=0.3mm] (3.5,-0.25) -- (3.5,0.25);
\draw[dashed] (3,0) -- (3,5);
\draw (5.5,0) -- (6.5,0);
\draw[line width=0.3mm] (5.5,-0.25) -- (5.5,0.25);
\draw[line width=0.3mm] (6.5,-0.25) -- (6.5,0.25);
\draw[dashed] (6,0) -- (6,5);
\draw (8.5,0) -- (9.5,0);
\draw[line width=0.3mm] (8.5,-0.25) -- (8.5,0.25);
\draw[line width=0.3mm] (9.5,-0.25) -- (9.5,0.25);
\draw[dashed] (9,0) -- (9,5);
\draw (11.5,0) -- (12.5,0);
\draw[line width=0.3mm] (11.5,-0.25) -- (11.5,0.25);
\draw[line width=0.3mm] (12.5,-0.25) -- (12.5,0.25);
\draw[dashed] (12,0) -- (12,5);
\node[scale=0.8] at (7.5,-1.5) {$\iota_n(v), \ v \in \pi_{n-1}^{-1}(w)$};
\end{scope}
\end{tikzpicture}
\end{center}
\caption{Illustration of the definition of $\iota_n$ on preimages 
of some vertex in $G_{n-1}$ when this vertex is not in a circuit corresponding to a finite orbit.\label{figure.cutting.process}}
\end{figure}
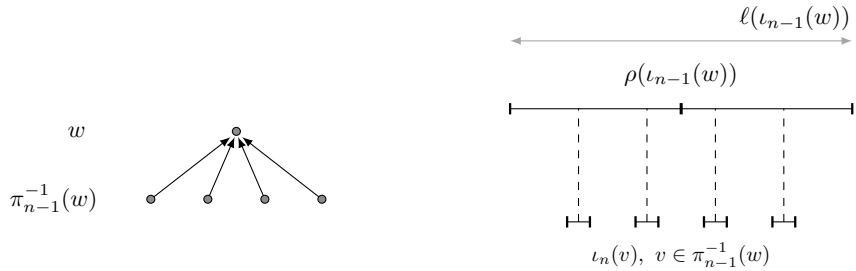

The choice of $\lambda_n$ and $\epsilon_n$ ensures that these intervals are 
disjoint and included in $\iota_{n-1}(w)$. Indeed the length of each of the intervals $\iota_n(v)$ for $v \in \pi_{n-1}^{-1} (w)$ is of the form $\lambda_n ^l \cdot \epsilon_n$ with $l \ge 0$ 
and thus is smaller than $\epsilon_n$. 
Since $\lambda_{n-1}^{|E_{n-1}|} \epsilon_{n-1}$ is smaller than the length of 
any interval $\iota_{n-1}(u)$ for $u \in V_{n-1}$, we have that for
$v \in \pi_{n-1}^{-1} (w)$
\[\ell(\iota_{n}(v)) \le \frac{1}{2} \cdot \frac{\ell(\iota_{n-1}(w))}{|V_n| +1} < \frac{1}{2} \cdot \frac{\ell(\iota_{n-1}(w))}{|\pi_{n-1}^{-1} (w)| +1}.\]
This implies that these intervals are disjoint and included in $\iota_{n-1}(w)$.

\item  \textbf{When $w$ is in a circuit included in the attracted part of $\mathcal{U}_{n-1}$:}

Let us denote by $p$ the finite orbit corresponding to this circuit.
Let us consider the intervals $I^{(n)}_k(w)$, $0 \le k \le \omega_n(p)$ such that for $k < \omega_n (p)$:
\begin{equation}
\rho(I^{(n)}_k(w)) = \rho(\iota_{n-1}(w)) - \frac{\ell(\iota_{n-1}(w))}{2} +  \frac{1}{2^{kn} } \cdot \frac{\ell(\iota_{n-1}(w))}{2}.\label{eq:rho_reg}
\end{equation}
In other words, this means that for $k < \omega_n(p)$ the distance between the center of the interval $I^{(n)}_k(w)$ and the leftmost point of the interval $\iota_{n-1}(w)$ is 
$\frac{1}{2} \cdot \frac{1}{2^{kn}} \cdot \ell(\iota_{n-1}(w))$.

When $k = \omega_n (p)$: 
\[\rho(I^{(n)}_k(w)) = \rho(\iota_{n-1}(w)) - \frac{\ell(\iota_{n-1}(w))}{2} + \frac{1}{8 \cdot 2^{n \omega_n (p)}} \cdot \ell(\iota_{n-1}(w)).\]

Moreover for all $k \le \omega_n(p)$, $\ell(I^{(n)}_k(w)) = 
\frac{1}{4 \cdot 2^{n \omega_n (p)}} \cdot \ell(\iota_{n-1}(w))$. 
These intervals are used as "containers" for the intervals $\iota_n(v)$ for 
$v \in \pi_{n-1}^{-1}(w)$. 

 In each of the intervals $I^{(n)}_k(w)$ we place the intervals 
 $\iota_n(v)$ for $v$ preimages of $w$ which are distance $k$ from the circuit, as follows. 
 For $j$
such that $v$ is the $j$th of elements of the set $S_k(w) = \{v \in \pi_{n-1}^{-1} (w) \ : \ \delta_n (v) = \omega_n(p) - k\}$:
\[\rho(\iota(v)) = \left(\rho(I^{(n)}_k(w)) - \frac{\ell(I^{(n)}_k(w))}{2}\right) + j \cdot \frac{\ell(I^{(n)}_k(w))}{|S_k(w)|+1}.\]
See an illustration on Figure~\ref{figure.cutting.process.cycles}.

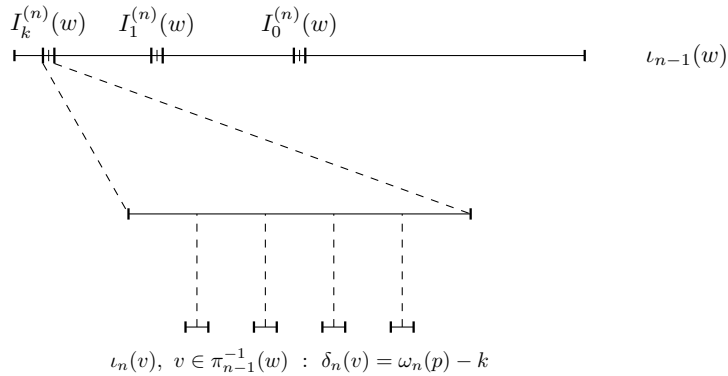
\begin{figure}[h!]
\begin{center}
\begin{tikzpicture}[scale=0.3]

\draw (-5,7) -- (20,7);
\draw[line width=0.3mm] (-5,7.25) -- (-5,6.75);
\draw[line width=0.3mm] (20,7.25) -- (20,6.75);
\draw[line width=0.1mm] (7.5,7.25) -- (7.5,6.75);
\draw[line width=0.3mm] (7.25,7.35) -- (7.25,6.65);
\draw[line width=0.3mm] (7.75,7.35) -- (7.75,6.65);
\node[scale=0.9] at (7.5,8.5) {$I^{(n)}_0(w)$};
\node[scale=0.9] at (-3.5,8.5) {$I^{(n)}_k(w)$};
\draw[dashed] (-3.75,6.65) -- (0,0);
\draw[dashed] (-3.25,6.65) -- (15,0);
\node[scale=0.9] at (1.25,8.5) {$I^{(n)}_1(w)$};
\draw[line width=0.1mm] (1.25,7.25) -- (1.25,6.75);
\draw[line width=0.3mm] (1,7.35) -- (1,6.65);
\draw[line width=0.3mm] (1.5,7.35) -- (1.5,6.65);
\draw[line width=0.1mm] (-3.5,7.25) -- (-3.5,6.75);
\draw[line width=0.3mm] (-3.75,7.35) -- (-3.75,6.65);
\draw[line width=0.3mm] (-3.25,7.35) -- (-3.25,6.65);
\node[scale=0.9] at (24.5,7) {$\iota_{n-1}(w)$};
\draw (0,0) -- (15,0);
\draw[line width=0.3mm] (0,-0.25) -- (0,0.25);
\draw[line width=0.3mm] (15,-0.25) -- (15,0.25);

\begin{scope}[yshift=-5cm]
\draw (2.5,0) -- (3.5,0);
\draw[line width=0.3mm] (2.5,-0.25) -- (2.5,0.25);
\draw[line width=0.3mm] (3.5,-0.25) -- (3.5,0.25);
\draw[dashed] (3,0) -- (3,5);
\draw (5.5,0) -- (6.5,0);
\draw[line width=0.3mm] (5.5,-0.25) -- (5.5,0.25);
\draw[line width=0.3mm] (6.5,-0.25) -- (6.5,0.25);
\draw[dashed] (6,0) -- (6,5);
\draw (8.5,0) -- (9.5,0);
\draw[line width=0.3mm] (8.5,-0.25) -- (8.5,0.25);
\draw[line width=0.3mm] (9.5,-0.25) -- (9.5,0.25);
\draw[dashed] (9,0) -- (9,5);
\draw (11.5,0) -- (12.5,0);
\draw[line width=0.3mm] (11.5,-0.25) -- (11.5,0.25);
\draw[line width=0.3mm] (12.5,-0.25) -- (12.5,0.25);
\draw[dashed] (12,0) -- (12,5);
\node[scale=0.8] at (7.5,-1.5) {$\iota_n(v), \ v \in \pi_{n-1}^{-1}(w) \ : \ \delta_n (v) = \omega_n(p)-k$};
\end{scope}
\end{tikzpicture}
\end{center}
\caption{Illustration of the definition of $\iota_n(v)$ for $v$ a preimage by $\pi_{n-1}$ of $w$ which is in a circuit of the attracted part of $\mathcal{U}_{n-1}$.\label{figure.cutting.process.cycles}}
\end{figure}
For similar reasons as in the first case, the intervals $I^{(n)}_k(w)$ are 
pairwise disjoint and included in $\iota_{n-1}(w)$, and the intervals $\iota_n (v)$ 
for $v \in \pi_{n-1}^{-1} (w)$ such that $\delta_n (w) = \omega_n(p) - k$ are disjoint 
and included in $I^{(n)}_k(w)$.
\end{enumerate}

\subsubsection{Embedding with vanishing derivative}

Since for all $n \ge 1$ and $v \in V_n$, 
\[\ell(\iota_n(v)) \le \epsilon_n \le \frac{1}{4^n},\]
and that for all $v$, $\iota_{n}(v) \subset \iota_{n-1} (\pi_{n-1}(v))$, 
for all $\textbf{v} \in V_{\textbf{G},\pi}$, the intersection $\bigcap_{n \ge 1} \iota_n (\textbf{v}_n)$ is 
reduced to a point in $\mathbb{R}$. Let us denote this point by $\psi(x)$, where $x$ is equal 
to $\varphi_{\mathbb{U},f}^{-1}(\textbf{v})$. By construction, $\psi$ is continuous and injective and thus a homeomorphism onto 
its image. Let us denote $Z= \psi(X)$. As a consequence $\psi$ conjugates $f$ with a map $\sigma : Z \rightarrow Z$. 

We are going to prove that for all $z \in Z$, $\sigma'(z)$ is defined and equal to $0$. Then a direct 
application of Theorem~\ref{theorem.jarnik} will end the proof of the theorem. 

Let us set $z=\psi(x)$.
We will prove that for all $n \ge 1$ and $\textbf{v} \in V_{\textbf{G},\pi}$, for $\textbf{v}'$ sufficiently close to $\textbf{v}$, we have:
\begin{equation}
	\left|\psi(f \circ \varphi_{\mathbb{U},f}^{-1}(\textbf{v}))-\psi(f \circ \varphi_{\mathbb{U},f}^{-1}(\textbf{v}'))\right|\le \frac{1}{2^n} \cdot \left|\psi( \varphi_{\mathbb{U},f}^{-1}(\textbf{v}))-\psi(\varphi_{\mathbb{U},f}^{-1}(\textbf{v}'))\right|
\label{eq:derrivativezeroestimate}
\end{equation}
which implies that
\[
\sigma'(z)=\sigma'(\psi(x))=\sigma'(\psi(\varphi_{\mathbb{U},f}^{-1}(\textbf{v})))=0.
\]
Let us consider some $\textbf{v} \in V_{\textbf{G},\pi}$ and $n \ge 1$. By construction 
there exists $m > n+3$ such that for all $k \ge m$, $\chi_k(\textbf{v}_k) \neq *$. Since 
divergent vertices coincide with markers, for all $k \ge m$, $\textbf{v}_k$ is not 
divergent.
Let us prove the above inequality when $\textbf{v}'$ coincides with 
$\textbf{v}$ on the $m$ first elements. Let us consider such sequence $\textbf{v}'$ 
and denote by $l$ the smallest among integers $k > m$ such that $\textbf{v}_k \neq \textbf{v}'_k$. From this point we distinguish two cases:

\begin{enumerate}
\item \textbf{The point $\textbf{v}$ is not periodic for the map $\underline{f}_{\mathbb{U}^{*}}$:} 
Thus we can assume without loss of generality that 
$\textbf{v}_{l-1} = \textbf{v}'_{l-1}$ is not in a circuit 
included in the attracted part of $(\mathcal{U}_{l-1},\tau_{l-1})$. 
As a consequence, and since $\textbf{v}_{l-1}$ is not divergent,
$\psi(f \circ \varphi_{\mathbb{U}^{*},f}^{-1}(\textbf{v}))$ and $\psi(f \circ \varphi_{\mathbb{U}^{*},f}^{-1}(\textbf{v}'))$ both lie in the same interval and this interval 
has (by definition of $\iota_{l-1}$) length bounded from above by $\lambda_{l-1} \cdot \ell(\iota_{l-1}(\textbf{v}_{l-1}))$. As a consequence: 
\[\left|\psi(f \circ \varphi_{\mathbb{U}^{*},f}^{-1}(\textbf{v}))-\psi(f \circ \varphi_{\mathbb{U}^{*},f}^{-1}(\textbf{v}'))\right| \le \lambda_{l-1} \cdot \ell(\iota_{l-1}(\textbf{v}_{l-1})).\]
By definition of $\lambda_{l-1}$ and $\rho$ we thus have: 

\[\lambda_{l-1} \cdot \ell(\iota_{l-1}(\textbf{v}_{l-1})) \le \frac{1}{2^{l-1}} \cdot \frac{\ell(\iota_{l-1}(\textbf{v}_{l-1}))}{
|V_l|+1} \le \frac{1}{2^{l-1}} \cdot \left|\rho(\iota_l (\textbf{v}_l))-\rho(\iota_l (\textbf{v}'_l))\right| \]
As a direct consequence: 
\[\left|\psi(f \circ \varphi_{\mathbb{U}^{*},f}^{-1}(\textbf{v}))-\psi(f \circ \varphi_{\mathbb{U}^{*},f}^{-1}(\textbf{v}'))\right| \le \frac{1}{2^{l-1}} \cdot \left|\rho(\iota_l (\textbf{v}_l))-\rho(\iota_l (\textbf{v}'_l))\right|.\]
Since $\psi (\varphi_{\mathbb{U}^{*},f}^{-1}(\textbf{v})) \in \iota_l (\textbf{v}_l)$ and $\psi (\varphi_{\mathbb{U}^{*},f}^{-1}(\textbf{v}')) \in \iota_l (\textbf{v}'_l)$ these points 
are at distance less than $\frac{1}{2} \ell(\iota_l (\textbf{v}_l))=\frac{1}{2} \ell(\iota_l (\textbf{v}'_l))$ from the respective points $\rho(\iota_l (\textbf{v}_l))$ and $\rho(\iota_l (\textbf{v}'_l))$. 
Recall that by the definition we have
$$
\left|\rho(\iota_l (\textbf{v}_l))-\rho(\iota_l (\textbf{v}'_l))\right| \ge \frac{\ell(\iota_{l-1} (\textbf{v}_{l-1}))}{|V_l|+1}
\ge 2\ell(\iota_{l} (\textbf{v}_l))
$$

and therefore
\begin{eqnarray*}
\left|\psi( \varphi_{\mathbb{U}^{*},f}^{-1}(\textbf{v}))-\psi(\varphi_{\mathbb{U}^{*},f}^{-1}(\textbf{v}'))\right| &\ge& \left|\rho(\iota_l (\textbf{v}_l))-\rho(\iota_l (\textbf{v}'_l))\right|-\ell(\iota_l (\textbf{v}_l))\\
&\ge& \frac{1}{2} \left|\rho(\iota_l (\textbf{v}_l))-\rho(\iota_l (\textbf{v}'_l))\right|.
\end{eqnarray*}

This gives
\[\left|\psi(f \circ \varphi_{\mathbb{U}^{*},f}^{-1}(\textbf{v}))-\psi(f \circ \varphi_{\mathbb{U}^{*},f}^{-1}(\textbf{v}'))\right| \le \frac{1}{2^{l-2}} \cdot \left|\psi( \varphi_{\mathbb{U}^{*},f}^{-1}(\textbf{v}))-\psi(\varphi_{\mathbb{U}^{*},f}^{-1}(\textbf{v}'))\right|\]
completing the proof of this case, since $l>m>n$, in particular $l-2\geq n$, and so \eqref{eq:derrivativezeroestimate} holds in this case.

\item \textbf{The point $\textbf{v}$ is periodic for the map $\underline{f}_{\mathbb{U}^{*}}$:}
In this case we can assume without loss of generality that 
$\textbf{v}_{l-1} = \textbf{v}'_{l-1}$ is in a circuit corresponding to a finite
orbit $p$ included in the attracted part of $(\mathcal{U}_{l-1},\tau_{l-1})$, and so is $\textbf{v}_l$. 
As a consequence $\textbf{v}'_l$ is not 
in this circuit (otherwise we would have $\textbf{v}_{l} = \textbf{v}'_{l}$ which contradicts the choice of $l$).

This implies that the point $\psi(\varphi_{\mathbb{U}^{*},f}^{-1}(\textbf{v}'))$ lies in some interval $I^{(l)}_k(\textbf{v}_{l-1})$ 
constructed in Section~\ref{section.setting} for $w=\textbf{v}_{l-1}$ and the integer $l$, for $k<\omega_l(p)$, where $p$ is the circuit to which $\textbf{v}_l$ belongs. The point $\psi(\varphi_{\mathbb{U}^{*},f}^{-1}(\textbf{v}))$ is the 
left extreme point of the interval $I^{(l)}_{\omega_l(p)} (\textbf{v}_{l-1})$. As a consequence the distance between the two points $\psi(\varphi_{\mathbb{U}^{*},f}^{-1}(\textbf{v}))$ and $\psi(\varphi_{\mathbb{U}^{*},f}^{-1}(\textbf{v}'))$ is larger than:

\[\left(\frac{1}{2} \frac{1}{2^{kl}} - \frac{1}{4 \cdot 2^{l\omega_l(p)}} - \frac{1}{8 \cdot 2^{l\omega_l(p)}} \right) \ell(\iota_{l-1}(\textbf{v}_{l-1})) \ge \frac{1}{8} \cdot \frac{1}{2^{kl}}  \ell(\iota_{l-1}(\textbf{v}_{l-1})).\]
With a similar reasoning the distance between $\psi(f \circ \varphi_{\mathbb{U}^{*},f}^{-1}(\textbf{v}))$ and $\psi(f \circ \varphi_{\mathbb{U}^{*},f}^{-1}(\textbf{v}'))$ is smaller 
than: 
\[\left(\frac{1}{2} \frac{1}{2^{(k+1)l}} + \frac{1}{8 \cdot 2^{l\omega_l(p)}}\right) \ell(\iota_{l-1}(\textbf{v}_{l-1})) \le \frac{5}{8 \cdot 2^{(k+1)l}} \ell(\iota_{l-1}(\textbf{v}_{l-1})).\]
By combining the equations we have that: 
\[\left|\psi(f \circ \varphi_{\mathbb{U}^{*},f}^{-1}(\textbf{v}))-\psi(f \circ \varphi_{\mathbb{U}^{*},f}^{-1}(\textbf{v}'))\right| \le 
\frac{5}{2^{l}}\cdot \left|\psi( \varphi_{\mathbb{U}^{*},f}^{-1}(\textbf{v}))-\psi(\varphi_{\mathbb{U}^{*},f}^{-1}(\textbf{v}'))\right|.\]
As a consequence, since $l > n+3$: 
\[\left|\psi(f \circ \varphi_{\mathbb{U}^{*},f}^{-1}(\textbf{v}))-\psi(f \circ \varphi_{\mathbb{U}^{*},f}^{-1}(\textbf{v}'))\right| \le \frac{1}{2^{n}} \cdot \left|\psi( \varphi_{\mathbb{U}^{*},f}^{-1}(\textbf{v}))-\psi(\varphi_{\mathbb{U}^{*},f}^{-1}(\textbf{v}'))\right|\]
which is the equation \eqref{eq:derrivativezeroestimate}.
\end{enumerate}

This concludes the proof of Theorem~\ref{theorem.main}.

\begin{remark}
In the first case of the above enumeration, we used critically the fact that divergent 
points coincide with markers. Without this, $\psi(f \circ \varphi_{\textbf{U},f}^{-1}(\textbf{v}))$ and $\psi(f \circ \varphi_{\textbf{U},f}^{-1}(\textbf{v}'))$ could actually 
belong to two different intervals. In this case we would not have a sufficiently 
small upper bound on their distance.
\end{remark}	

\section*{Acknowledgements}
This work was supported by National Science Centre, Poland (NCN),\break grant no. 2019/35/B/ST1/02239.

\end{document}